\documentclass[smallextended,numbook,runningheads]{svjour3}
\usepackage{amsmath}
\smartqed  % flush right qed marks, e.g. at end of proof
\usepackage{amsfonts,amssymb}
\usepackage{hyperref}
\usepackage{epsfig}
\usepackage{booktabs}
\usepackage{natbib}
\usepackage{xcolor}
\usepackage[capitalise]{cleveref}
\usepackage{epstopdf}

\journalname{}
\date{ \phantom{b} \vspace{45mm}\phantom{e}}
\def\makeheadbox{\hspace{-4mm}Version of 18 March 2026}%, file: adiabatic.tex}

\definecolor{refblue}{rgb}{0,0,0.75} 
\definecolor{refblueb}{rgb}{0,0,1} 
\definecolor{refgreen}{rgb}{0,0.5,0} 
\definecolor{grey}{rgb}{0.7,0.7,0.7} 

\usepackage{hyperref}
\hypersetup{ 
	colorlinks   = true, %Colours links instead of ugly boxes 
	urlcolor     = refblue, %Colour for external hyperlinks 
	linkcolor    = refblueb, %Colour of internal links 
	citecolor   = refblue%Colour of citations 
}

\def\param{{\textcolor{black}{\theta}}}
\def\paramSpace{{\textcolor{black}{\mathcal{Q}}}}

\def\R{{\mathbb R}}

\def\C{{\mathbb C}}

\def\e{{\mathrm e}}

\def\iu{\mathrm{i}}
\def\eps{\varepsilon}

\def\calH{{\cal H}}

\def\calM{{\cal M}}

\def\wt{\widetilde}
\def\wh{\widehat}

\def\Re{{\mathrm{Re}\,}}
\def\Im{{\mathrm{Im}\,}}

\newdimen\GGGlength
\newdimen\GGGheight
\newbox\GGGbox
\def\GGGput[#1,#2](#3,#4)#5{%
  \setbox\GGGbox\vbox{\hbox{#5}\kern0pt}%
  \GGGlength\wd\GGGbox%
  \divide\GGGlength by100 \multiply\GGGlength by#1%
  \GGGheight\ht\GGGbox%
  \divide\GGGheight by100 \multiply\GGGheight by#2%
  \put(#3,#4){\kern-\GGGlength\raise-\GGGheight\box\GGGbox}}

\newcommand{\bch}{\color{black}}
\newcommand{\ech}{\color{black}}

\begin{document}

\title{Regularized dynamical parametric approximation}

\titlerunning{Regularized dynamical parametric approximation}

\author{ }
\authorrunning{ }

\author{Michael Feischl$^1$, Caroline Lasser$^2$, Christian Lubich$^3$, Jörg Nick$^4$}
\authorrunning{M.\ Feischl, C.\ Lasser, C.\ Lubich, J.\ Nick}

\institute{
$^1$~Institute for Analysis and Scientific Computing, Wiedner Hauptstra\ss e 8--10, 1040 Wien, TU Wien, Austria.
\email{michael.feischl@tuwien.ac.at}\\
$^2$~Department of Mathematics, Boltzmannstra\ss e 3, TU M\"unchen, D-85748 Garching bei M\"unchen, Germany.
\email{classer@ma.tum.de}\\
$^3$~Mathematisches Institut, Auf der Morgenstelle 10, Univ.\ T\"ubingen, D-72076 T\"ubingen, Germany.
\email{Lubich@na.uni-tuebingen.de}\\
$^4$~Seminar für Angewandte Mathematik, Rämistrasse 101, ETH Z\"urich, CH-8092 Z\"urich, Switzerland.
\email{joerg.nick@math.ethz.ch}\\
}

\date{ }

\maketitle

 \begin{abstract} 
This paper studies the numerical approximation of evolution equations by nonlinear parametrizations $u(t)=\Phi(\param(t))$ with time-dependent parameters $\param(t)$, which are to be determined in the computation. The motivation comes from approximations in quantum dynamics by multiple Gaussians and approximations of various dynamical problems by tensor networks and neural networks. In all these cases, the parametrization is typically irregular: the derivative $\Phi'(\param)$ can have arbitrarily small singular values and may have varying rank. We derive approximation results for a regularized approach in the time-continuous case as well as in time-discretized cases. For the latter, there is a nontrivial interplay between the regularization parameter and the time stepsize, depending also on the defect size and local bounds of the second derivative of the parametrization map $\Phi$. When this is appropriately taken into account, the approach can be successfully applied in irregular situations, even though it runs counter to the basic principle in numerical analysis to avoid solving ill-posed subproblems when aiming for a stable algorithm. 
The paper also includes two theoretical case studies: regularized parametric steepest descent for optimization and regularized parametric Strang splitting for the time-dependent Schr\"odinger equation.
Numerical experiments with sums of Gaussians for approximating quantum dynamics and with neural networks for approximating the flow map of a system of ordinary differential equations illustrate and complement the theoretical results.

\medskip\noindent
{\it Keywords.\,} Time-dependent nonlinear parametric approximation, regularization, high-dimensional differential equation, gradient flow, Schr\"odinger equation, time integration, deep neural network, multi-Gaussian approximation.

\end{abstract}

\tableofcontents

\section{Introduction}

\subsection{Nonlinear parametric approximation of evolution problems}
We consider the numerical approximation of a possibly high-dimensional initial-value problem of ordinary or partial differential equations
\begin{equation}\label{ivp}
\dot y = f(y), \qquad y(0)=y_0
\end{equation}
via a nonlinear parametrization
\begin{equation}\label{Phi}
y(t) \approx u(t)=\Phi(\param(t)), \qquad 0\le t \le \overline t,
\end{equation}
with time-dependent parameters $\param(t)$. Here, $\Phi$ is a smooth map from a parameter space into the solution space. We are interested in the situation of irregular parametrizations: the derivative matrix $\Phi'(\param)$ may have arbitrarily small singular values and possibly also varying rank. This is a typical situation of over-approximation that frequently arises in applications and causes severe numerical difficulties.
Our motivation for studying such problems originated from the following areas, for which 
references are given and discussed in Subsection~\ref{subsec:lit}.
\begin{itemize}
\item 
Multi-Gaussian approximations in quantum dynamics: Here the parameters are the evolving complex width matrices, positions, momenta, and phases in a sum of complex Gaussians.
\item 
Tensor network approximations in quantum dynamics: Here the parameters are the evolving connection tensors and bases.
\item
Deep neural network approximations of dynamical problems: 
Here the parameters are the evolving weight matrices and biases of the various layers of the neural network.
\end{itemize}
In all these cases, the derivative matrix $\Phi'(\param)$ typically has numerous very small singular values.

\subsection{Regularized dynamical parametric approximation}
In the numerical approach considered in this paper, we determine the time derivatives $\dot \param(t)$ and 
$\dot u(t)=\Phi'(\param(t))\dot \param(t)$  by solving the regularized linear least squares problem (we omit the argument $t$)
\begin{equation}
\label{reg-lsq-intro}
\| \dot u - f(u) \|^2 +\eps^2 \| \dot \param \|^2 = \min !
\end{equation}
with a possibly time-dependent regularization parameter $\eps(t)>0$. This yields a differential equation for the parameters $\param$, and then $u=\Phi(\param)$. We refer to this approximation as a  regularized dynamical parametric approximation. 
The resulting differential equation for $\param$ is solved numerically by a standard time-stepping method, where each function evaluation solves a linear least squares problem \eqref{reg-lsq-intro}. 

This algorithmic approach is studied here even though the differential equation for the parameters $\param$ is severely ill-conditioned for small $\eps$. Errors in the initial value $\param(0)$ can increase by a factor $e^{ct/\eps}$ (with a constant $c>0$ independent of $\eps$) to errors in $\param(t)$ at times $t>0$.
The approach thus appears to run counter to a basic principle of numerical analysis:  to avoid solving ill-posed subproblems when aiming for a stable algorithm.

As a consequence of the ill-posedness in the parameters, also the error propagation in the solution approximation $u(t)=\Phi(\param(t))$ is ill-behaved, but nevertheless the problem turns out to be what might be called {\it well-posed up to the order of the defect size} $O(\delta)$, where $\delta=\max_{0\le t \le \bar t}\delta(t)$ and $\delta(t)^2$ is the minimum value attained in \eqref{reg-lsq-intro} at time $t$. This beneficial behaviour is found both in the differential equation that results from \eqref{reg-lsq-intro} and in its numerical time discretization, and this makes the regularization \eqref{reg-lsq-intro} a viable computational approach. Its analysis adds new aspects to that of the underlying differential equation \eqref{ivp} and the time discretization.

One possibly obvious finding from our analysis is that good pointwise approximability of the solution $y(t)$ by parametrized functions is not sufficient. It is important that the time derivative $\dot y(t)$ can be well approximated in the tangent spaces at parametrized functions $u=\Phi(\param)$ near $y(t)$. This suffices when $f$ is Lipschitz-continuous. In the case of a dominant term $Ly$ in $f(y)=Ly+g(y)$ with an unbounded operator $L$ that maps parametrized functions into their tangent space (a situation encountered for the Schr\"odinger equation), we need that $\dot y(t) - Ly(t)$ can be well approximated in the tangent space, and the
regularized least squares problem \eqref{reg-lsq-intro} should be slightly modified.

\begin{remark} [Truncation vs.~regularization]
As an alternative to regularizing the least squares problem as in \eqref{reg-lsq-intro}, small singular values of the matrix $A=\Phi'(\param)$ below a threshold $\eps>0$ are set to zero, yielding a truncated matrix $A_\eps$ with smallest nonzero singular value at least~$\eps$. Instead of \eqref{reg-lsq-intro}, the time derivative $\dot \param$ is then determined to be of minimal norm such that
\begin{equation}\label{trunc-lsq}
\| A_\eps(\param)\dot \param - f(\Phi(\param)) \|^2  = \min !
\end{equation}
Because of discontinuities in $A_\eps(\cdot)$ when singular values cross the threshold, the resulting differential equation $\dot \param = A_\eps(\param)^+f(\Phi(\param))$ with the Moore--Penrose pseudo-inverse $A_\eps^+$ has a discontinuous right-hand side and is problematic to analyse.
After time discretization, however, this approach can be analysed by arguments that are analogous to the regularized approach \eqref{reg-lsq-intro} and it is found to behave in a very similar way, since $A_\eps^+$ behaves similarly to the matrix $(A^\top A+ \eps^2 I)^{-1} A^\top$ that appears in the normal equations for~\eqref{reg-lsq-intro}. Both matrices have the same singular vectors and the singular values are $\sigma_i^{-1}$ if $\sigma_i\ge \eps$ and zero else, and $(\sigma_i^2+\eps^2)^{-1}\sigma_i$, respectively, where $\sigma_i$ are the singular values of $A$. \textcolor{black}{Therefore, we observe very similar numerical behaviour of the two approaches, but the computational cost of truncation is usually significantly higher compared to regularization.}
\end{remark}

\subsection{Related work} \label{subsec:lit}

In quantum dynamics, the non-regularized approach \eqref{reg-lsq-intro} with $\eps=0$ is known as the Dirac--Frenkel time-dependent variational principle and has been widely used in computational chemistry and quantum physics over the past decades; see the books \cite{kramer1981,Lub08} for dynamic, geometric and approximation aspects and among countless papers see e.g. \cite{Joubert24,worth2004novel,richings2015quantum} with Gaussian wavepackets, \cite{meyer1990multi,wang2003multilayer} and \cite{shi2006classical,haegeman2016unifying} with tensor networks and 
\cite{carleo2017solving,hartmann2019neural,schmitt2020quantum} with artificial neural networks. This approach usually leads to ill-conditioned least squares problems which cause numerical difficulties\footnote{The case of tree tensor networks, which includes low-rank matrices, Tucker tensors and tensor trains as special cases, is exceptional since its multilinear geometry allows for time integration algorithms that are robust to arbitrarily small singular values \cite{ceruti2021time,lubich2015time}.}. 
{\it Ad hoc} regularization is often used in computations in such cases, but a systematic foundation and analysis appear to be missing in the literature. The regularization chosen in \cite{richings2015quantum} corresponds precisely to \eqref{reg-lsq-intro} but is not further investigated there.

In recent years, an extensive amount of research has been invested in nonlinear approximations of differential equations, in particular in the context of neural networks. The majority of the literature relies on minimizing the residual in space-time, which yields a nonlinear optimization problem for the parameters globally in time. Prominent examples of such a strategy include the ``deep Ritz method" \cite{BW18} as well as ``weak generative adversarial networks"~\cite{ZBYZ20} and both were successfully applied to a wide range of partial differential equations (see e.g. \cite{ZZKP19,BUMM21,BBGJJ21,KZBKKAA23}). The nonlinear optimization process is, however, difficult to analyse and the existing theory consequently focuses on the expressivity properties of neural networks, i.e.,
what approximation is theoretically possible (see e.g. \cite{OPS20,MM23,KPRS2022}). This of course excludes the much harder task of training the network and explaining the success of these methods remains evasive.

An alternative approach in the literature is to use Galerkin based methods to obtain differential equations for the parameters. 
The Dirac--Frenkel time-dependent variational principle has been used outside the realm of quantum dynamics in the context of deep neural networks, for example in \cite{DZ21} (referred to as ``evolutional deep neural network") and in \cite{BPV24} (referred to as ``neural Galerkin schemes"). 

The work~\cite{KO2001} analyzes and proposes several regularization schemes in the context of Krylov-subspace projections of the original problem. This is of interest if one additionally wants to analyze the solution process of the arising linear systems, which we do not study here.

Since the original submission of this paper in March 2024, several further related articles have appeared. In particular, we mention \cite{BCM25}, which discusses dynamic sampling strategies for parametric approaches and investigates their application to gradient flows. Further we mention \cite{DSP25}, which mitigates the effect of the ill-conditioned least squares problems by a randomization strategy.

\subsection{Outline}

\bigskip
After establishing notation and framework in Section~\ref{sec:regdyn}, we derive {\it a posteriori} and {\it a priori} error bounds for the time-continuous regularized approach \eqref{reg-lsq-intro} in Section~\ref{sec:err}, where we also study the sensitivity of approximations to initial values. 

In Section~\ref{sec:time-disc} we study the error behaviour of the time discretization by the explicit and implicit Euler methods, where we observe in particular the interplay between the time stepsize and the regularization parameter. In Section~\ref{sec:rk} the result is extended  to Runge-Kutta discretizations, for which we prove an optimal-order error bound up to the defect size $\delta$. In Section~\ref{sec:eps-h}
we propose an adaptive selection of the regularization parameter and the stepsize that is based on the previous error analysis. 

In Section~\ref{sec:con} we study the effect of enforcing conserved quantities, which are typically not preserved by the regularized parametric approach.

In Section~\ref{sec:grad-sys} we study an approach to minimize a function via regularized parametric steepest descent,
which corresponds to using the regularized parametric Euler method for the gradient system. The resulting parametric algorithm retains the decay properties of non-parametric steepest descent up to the defect size. It converges to the global minimum of strongly convex functions up to an error of the defect size, unlike gradient descent for the parameters.

In Section~\ref{sec:psi} we study the  regularization approach for the time-dependent Schr\"odinger equation as an important exemplary case of an evolutionary partial differential equation. We give an error analysis of the regularized parametric version of the Strang splitting method.

Finally, in Section~\ref{sec:num} we present two numerical examples that illustrate the behaviour of the regularized parametric approximation: first by a neural network to approximate the flow map of a system of ordinary differential equations, and second by a linear combination of complex Gaussians to approximate quantum dynamics.

In an appendix we collect some useful results on regularized least squaresproblems.

\section{Regularized dynamical parametric approximation} 
\label{sec:regdyn}
We start with the formulation of the framework.
Let the state space $\calH$ be a Hilbert space (finite- or infinite-dimensional) with the inner product $\langle\cdot,\cdot\rangle_\calH$ and corresponding norm $\|\cdot\|_\calH$ and let the parameter space $\paramSpace$ be a finite-dimensional vector space equipped with the inner product $\langle\cdot,\cdot\rangle_\paramSpace$ and corresponding norm $\|\cdot\|_\paramSpace$.
The parametrization map $\Phi:\paramSpace\to\calH$ is assumed to be twice continuously differentiable. It need not be one-to-one, not even locally. The derivative $\Phi'(\param)$ may have arbitrarily small singular values and a nontrivial nullspace of varying dimension. 

The initial value in \eqref{ivp} is assumed to be in parametrized form: $y(0)=\Phi(\param(0))$ for some $\param(0)\in\paramSpace$. We take $u(0)=y(0)=\Phi(\param(0))$.

For $t\ge 0$, let $\eps(t)>0$ be the regularization parameter.
For $u(t)\in\calH$ and $\param(t)\in\paramSpace$ with $u(t)=\Phi(\param(t))$,  the time derivatives $\dot u(t)=\Phi'(\param(t))\dot \param(t)\in\calH$ 
and $\dot \param(t)\in \paramSpace$ are determined by requiring that
\begin{equation}
\label{reg-lsq}
\delta(t)^2 := \| \dot u(t) - f(u(t)) \|_\calH^2 +\eps(t)^2 \| \dot \param(t) \|_\paramSpace^2 \quad\text{is minimal.}
\end{equation}
In this regularized linear least squares problem, $\dot \param(t)$ is uniquely determined, which it would not necessarily be without the regularization. Problem \textcolor{black}{\eqref{reg-lsq} can equivalently be written in the damped form 
\begin{equation}\label{reg-damp}
\left\| \begin{pmatrix}
     \Phi'(\param(t)) \\ \eps(t) \,\mathrm{Id}_\paramSpace\\
\end{pmatrix}\dot \param(t) - \begin{pmatrix} f(u(t))\\ 0\end{pmatrix} \right\|_{\calH\times\paramSpace}  \quad\text{is minimal}
\end{equation}
and solved by well-established methods, see e.g. \cite{Bjoerck24}. The normal equations\footnote{If $\calH$ is infinite-dimensional, the linear map $A:\R^k\to\calH$ can be viewed as a quasi-matrix $A=(a_1,\dots,a_k)$ with $a_i\in \calH$, see e.g. \cite{townsend2015continuous}, for which $A^\top$ is to be interpreted as the adjoint of $A$. For $v,w\in \calH$ we interpret $v^\top w$ as the inner product $(v,w)_\calH$.
In the following we will use the familiar matrix notation throughout.} are}
$$
M_\eps(\param) \dot \param = A(\param)^\top f(\Phi(\param))
$$
with the symmetric positive matrix $M_\eps=A^\top A + \eps^2 Q$, where $A=\Phi'$ and the symmetric positive definite matrix $Q$ defines the norm $\|\param\|_\paramSpace=(\param^\top Q\param)^{1/2}$.
With the quasi-projection $P_\eps(\param) = A(\param) M_\eps(\param)^{-1} A(\param)^\top$,
we then have
$$
\dot u = P_\eps(\param) f(u).
$$
\ech
\begin{remark}
\bch
We  give an interpretation of \eqref{reg-lsq} as a projection. 
\ech
The graph
$$
\calM=\{ (u,\param)\,:\, u =\Phi(\param) \} \subset \calH\times\paramSpace
$$
need not be a manifold, but at $(u,\param)\in\calM$ it has the tangent space 
$$
T_{(u,\param)}\calM = \{ (\dot u,\dot \param) \,:\,  \dot u = \Phi'(\param)\dot \param \} \subset \calH \times \paramSpace.
$$
With respect to the $\eps$-weighted inner product on $\calH \times \paramSpace$ that is defined by 
$\langle\cdot,\cdot\rangle_\calH+ \eps^2 \langle\cdot,\cdot\rangle_\paramSpace$, we consider the orthogonal projection
onto $T_{(u,\param)}\calM$, denoted
$$
\Pi_{(u,\param)}^\eps: \calH \times \paramSpace \to T_{(u,\param)}\calM.
$$
Then, \eqref{reg-lsq} is equivalent to (omitting the argument $t$)
\begin{equation}\label{P-eq}
(\dot u, \dot \param)=\Pi_{(u,\param)}^\eps\bigl( f(u),0 \bigr).
\end{equation}
While this reformulation is reminiscent of the useful interpretation of the Dirac--Frenkel time-dependent variational principle as a projection, see \cite[Section II.1]{Lub08}, the difficulty here is that $\Pi_{(u,\param)}^\eps$ has no moderate Lipschitz constant with respect to $(u,\param)$ under our assumptions, as is quantified in Proposition~\ref{prop:sensitivity} below. The Lipschitz-conti\-nuity of the projection onto the tangent space was an essential condition in the proof of quasi-optimality of the approximation obtained from the Dirac--Frenkel time-dependent variational principle in \cite[Section II.6]{Lub08}. This is not available here, and so the question arises as to what alternative kind of error analysis can still be done. 
\end{remark}

\section{Error bounds in the time-continuous setting}
\label{sec:err}

In this section we assume that the vector field $f:\calH\to\calH$ in the differential equation \eqref{ivp} is  a continuous map with the one-sided Lipschitz constant $\ell$: for all $v,w\in \calH$,
\begin{equation}
\label{Lip}
%\begin{aligned}
\Re\langle v-w, f(v)-f(w) \rangle_\calH \le \ell\, \|v-w\|_\calH^2, \\
% \|f(v)-f(w)\|_\calH &\le L \,\|v-w\|_\calH.
%\end{aligned}
\end{equation}
The real part is taken in the case of a complex-linear space $\calH$.
In some of the following results we require $f$ to be Lipschitz-continuous with the Lipschitz constant $L$\,: for all $v,w\in \calH$,
\begin{equation}
\label{Lip-L}
% \begin{aligned}
% \Re\langle v-w, f(v)-f(w) \rangle_\calH \le \ell\, \|v-w\|_\calH^2, \\
 \|f(v)-f(w)\|_\calH \le L \,\|v-w\|_\calH.
%\end{aligned}
\end{equation}

\subsection{A posteriori error bound}
For the defect $d(t)$ of the approximation $u=\Phi(\param)$ we have 
\begin{equation}\label{defect}
\dot u(t) = f(u(t)) + d(t) 
\qquad\text{with}\quad \|d(t)\|_\calH \le \delta(t)
\end{equation}
for $\delta(t)$ of \eqref{reg-lsq}. There is the following error bound in terms of $\delta$.

\begin{proposition} \label{prop:delta}
In the situation of Section~\ref{sec:regdyn} and with the one-sided Lip\-schitz condition \eqref{Lip}, the error is bounded by
$$
\| u(t)-y(t) \|_\calH \le 
\textcolor{black}{e^{\ell t} \, \|u(0) - y(0)\|_{\mathcal{H}}} +
\int_0^t  e^{\ell(t-s)}\, \delta(s)\, ds.
$$
\end{proposition}

\begin{proof} The proof is standard and is included for the convenience of the reader.
We subtract \eqref{ivp} from \eqref{defect} and take the inner product with $u-y$ and its real part (if $\calH$ is a complex space). On the left-hand side this yields, omitting the omnipresent argument $t$,
$$
\Re\langle u-y, \dot u - \dot y \rangle = \frac12 \frac d{dt}\| u-y \|^2 = \|u-y\|\cdot \frac d{dt}\| u-y \|
$$
and on the right-hand side 
$$
\Re\langle u-y,f(u)-f(y)+d \rangle \le \|u-y\| \bigl( \ell \, \|u-y\| + \delta \bigr).
$$
Dividing both sides by $\|u-y\|$ and using Gronwall's inequality yields the result.
\qed
\end{proof}

While $\delta(t)$ can be monitored during the computation and is thus available {\it a posteriori}, we have not bounded it {\it a priori} from known or assumed approximability properties of the exact solution. This is done next.

\subsection{\textcolor{black}{Quasi-optimality of the approximation}}
We will bound the defect size $\delta(t)$ by a quantity that measures the approximability of the solution derivative $\dot y(t)$ in the tangent spaces $T_{(u,\param)}\calM$ for all $u=\Phi(\param)$ in a neighbourhood of $y(t)$. We fix a radius $\rho>0$ such that for every $t\le \bar t$, there exist parametrized $u=\Phi(\param)$ with $\|u-y(t)\|_\calH \le \rho$, and define
\begin{equation}
    \label{delta-rho}
    \bar\delta_\rho(t)^2 := \sup_{\param\in\paramSpace: \| \Phi(\param)-y(t) \|_\calH \le \rho}\ 
    \min_{\dot \param \in \paramSpace} \Bigl( \|\Phi'(\param)\dot \param - \dot y(t) \|^2 + \eps^2 \| \dot \param \|^2 \Bigr).
\end{equation}
This quantity measures the relevant approximation capability (or expressivity in another terminology) of the parametrization. It depends on the exact solution $y(t)$ of \eqref{ivp} but is independent of the regularized parametric approximation $u(t)=\Phi(\param(t))$ obtained from \eqref{reg-lsq}.
We can bound $\delta(t)$ of \eqref{reg-lsq}, which does depend on the desired approximation $u(t)=\Phi(\param(t))$, in terms of $\bar\delta_\rho(t)$.

\begin{lemma}
\label{lem:delta-bound}
Provided that $\|u(t)-y(t)\| \le \rho$, we have
$$
\ \delta(t) \le \bar\delta_\rho(t) + \| f(u(t))-f(y(t)) \|_\calH.
$$
\end{lemma}

\begin{proof} 
We omit the argument $t$ in the following. We have $u=\Phi(\param)$ and $\dot u=\Phi'(\param)\dot \param$.
Let $\dot \param_+\in \paramSpace$ be such that 
$$
\|\Phi'(\param)\dot \param_+ - \dot y \|_\calH^2 + \eps^2 \| \dot \param_+ \|_\paramSpace^2 \quad\text{is minimal.}
$$
We obtain from \eqref{reg-lsq}, \eqref{delta-rho}  and $\dot y=f(y)$ 
\begin{align*}
\delta^2 &=  \| \Phi'(\param)\dot \param - f(u) \|_\calH^2 + \eps^2 \| \dot \param \|_\paramSpace^2
\le \| \Phi'(\param)\dot \param_+ - f(u) \|_\calH^2 + \eps^2 \| \dot \param_+ \|_\paramSpace^2
\\
&\le \Bigl( \| \Phi'(\param)\dot \param_+ - \dot y \|_\calH + \| f(y)-f(u) \|_\calH \Bigr)^2 + \eps^2 \| \dot \param_+ \|_\paramSpace^2
\\
&\le {\bar\delta_\rho}^{\,2} + 2 \bar\delta_\rho \, \| f(y)-f(u) \|_\calH + \| f(y)-f(u) \|_\calH ^2
\\
&= \Bigl( \bar\delta_\rho + \| f(y)-f(u) \|_\calH \Bigr)^2.
\end{align*}
Taking square roots yields the result.
\qed
\end{proof}

We construct a reference approximation $u_*(t)=\Phi(\param_*(t))$ with $\param_*(t)\in \paramSpace$ from the exact solution such that its derivative is a regularized best approximation in the tangent space at $(u_*(t),\param_*(t))$ to the solution derivative $\dot y(t)$. The derivatives $\dot u_*(t)= \Phi'(\param_*(t))\dot \param_*(t)$ and $\dot \param_*(t)$ are determined such that
\begin{equation}\label{reg-lsq-star}
\delta_*(t)^2 := \| \dot u_*(t)-\dot y(t) \|_\calH^2 + \eps^2 \| \dot \param_*(t) \|_\paramSpace^2 \quad\text{is minimal}.
\end{equation}
Here we have the immediate error bound
\begin{equation}\label{err-u-star}
\| u_*(t)-y(t) \|_\calH %=\Bigl\| \int_0^t  \bigl( \dot u_*(s)-\dot y(s) \bigr) \, ds \Bigr\|_\calH
\le \int_0^t  \| \dot u_*(s)-\dot y(s) \|_\calH \, ds \le \int_0^t  \delta_*(s)\, ds  \le \int_0^t  \bar\delta_\rho(s)\, ds
\end{equation}
as long as this bound does not exceed $\rho$.
The following result bounds the error of the numerical approximation $u(t)$ by a multiple of the bound in \eqref{err-u-star}.
\begin{proposition}
\label{prop:delta-star}
In the situation of Section~\ref{sec:regdyn} and under the Lipschitz condition \eqref{Lip}, the error is bounded by
$$
\| u(t)-y(t) \|_\calH \le
\textcolor{black}{e^{(\ell+L)t} \, \|u(0) - y(0)\|_{\mathcal{H}}} +
e^{(\ell+L) t}\int_0^t   \bar\delta_\rho(s)\, ds
$$
as long as this does not exceed $\rho$.
%where $L_+=L + \max\{\ell,0\}$.
\end{proposition}

\begin{proof}
The bound follows from Lemma~\ref{lem:delta-bound} inserted into the proof of Proposition~\ref{prop:delta} and using the Lipschitz condition on $f$ and the Gronwall lemma.
\qed
\end{proof}

We do not know if there exists a similar bound that is independent of the Lipschitz constant $L$, which does not appear in the bound of Proposition~\ref{prop:delta}.

\subsection{Sensitivity to initial values}
Given two initial values $y_0,\wt y_0 \in \calH$, the difference of the correponding solutions $y(t)$ and $\wt y(t)$ of the differential equation \eqref{ivp} are bounded, under the one-sided Lipschitz condition \eqref{Lip}, by
$$
\| y(t)-\wt y(t) \|_\calH \le e^{\ell t} \,\| y_0-\wt y_0 \|_\calH, \qquad t\ge 0;
$$
see e.g. \cite[IV.12]{HW}. There is no analogous result for the regularized approximations $u(t)=\Phi(\param(t))$ and $\wt u(t)=\Phi(\wt \param(t))$ with initial values $u_0 = \Phi(\param_0)$ and $\wt u_0 = \Phi(\wt \param_0)$, not even with the Lipschitz constant $L$ instead of $\ell$.
We only obtain the following bound.

\begin{proposition}\label{prop:sensitivity-delta}
Under the one-sided Lipschitz condition \eqref{Lip} we have
\begin{equation}\label{u-diff}
\| u(t)-\wt u(t) \|_\calH \le e^{\ell t}\, \| u_0-\wt u_0 \|_\calH 
+ \int_0^t  e^{\ell(t-s)} \bigl( \delta(s)+\wt \delta(s)\bigr) ds, \quad\ t\ge 0,
\end{equation}
with the defect size $\delta(t)$ of \eqref{reg-lsq} and analogously $\wt \delta(t)$ corresponding to $\wt u(t)$. 
\end{proposition}

\begin{proof}
    This bound is obtained by the argument of Proposition~\ref{prop:delta} and estimating the difference of the defects $d(t)=\dot u(t)-f(u(t))$ of \eqref{defect} and analogously $\wt d(t)$ in a rough way by
$\| d(t)- \wt d(t) \|_\calH \le \delta(t)+\wt \delta(t)$.
\qed
\end{proof}

Note that the right-hand side of \eqref{u-diff} does not tend to zero as $u_0-\wt u_0\to 0$.
The difference of the defects is bounded in this rough way because there does not seem to be a better way that does not introduce negative powers of $\eps$ into the bound.
The difficulty becomes apparent in the following bound for $\dot u(t)- \dot {\wt u}(t)$, where the second term on the right-hand side is a multiple of 
$\|\param(t)-\wt \param(t)\|_\paramSpace$, which cannot be avoided and cannot be estimated by an $\eps$-independent multiple of $\|u(t)-\wt u(t)\|_\calH$, as we would like.
\begin{proposition}\label{prop:sensitivity}
As a function of $\param$, we denote by $\dot \param(\param)$ the solution of the regularized least squares problem
\begin{equation}\label{lsq-q}
\delta(\param)^2 = \| \Phi'(\param)\dot \param(\param) - f(\Phi(\param)) \|_\calH^2 +\eps^2 \| \dot \param(\param) \|_\paramSpace^2 \quad\text{is minimal}.
\end{equation}
Then, for all $\param,\wt \param\in\paramSpace$ and associated $u=\Phi(\param)$, $\wt u = \Phi(\wt \param)$ and 
$\dot u=\Phi'(\param)\dot \param(\param)$, $\dot {\wt u}=\Phi'(\wt \param)\dot \param(\wt \param)$,
\begin{align*}
\|\dot u - \dot{\wt u}\|_\calH 
&\le  L\|u-\wt u\|_\calH + \tfrac52\beta\,\frac{\bar\delta}{\eps}\, \|\param-\wt \param\|_\paramSpace ,
\\
\|\eps\dot \param(\param) - \eps\dot \param(\wt \param)\|_\paramSpace 
&\le  L\|u-\wt u\|_\calH + \tfrac52\beta\,\frac{\bar\delta}{\eps}\, \|\param-\wt \param\|_\paramSpace,
\\
and \qquad
|\delta(\param)-\delta(\widetilde \param)|&\le %5 \beta \frac{\bar\delta}\eps \, \| \param-\wt \param \|_\paramSpace + 3L \| \Phi(\param) - \Phi(\wt \param) \|_\calH,
3L \| u-\wt u \|_\calH + 5 \beta \frac{\bar\delta}\eps \, \| \param-\wt \param \|_\paramSpace,
\end{align*}
where $\beta$ is an upper bound on second order derivatives of $\Phi$ and $\bar\delta$ an upper bound on $\delta$ on the line segment between $\param$ and $\wt \param$, and $L$ is the Lipschitz constant of $f$ in \eqref{Lip}.
\end{proposition}

\begin{proof}
We use the adjoint $\Phi'(\param)^*\in L(\calH,\paramSpace)$ of the linear map $\Phi'(\param)\in L(\paramSpace,\calH)$ and the normal equations $M_\eps(\param)\dot \param(\param) = \Phi'(\param)^*f(\Phi(\param))$ with Gramian matrix 
$M_\eps(\param)=\Phi'(\param)^*\Phi'(\param)+\eps^2 I_\paramSpace$. We denote
\begin{align*}
\Phi'(\param)\dot \param(\param) &= \Phi'(\param)M_\eps(\param)^{-1}\Phi'(\param)^*f(\Phi(\param)) \\
&=: P_\eps(\param) f(\Phi(\param)).
\end{align*}
By Lemma~\ref{lem:P-bound}, we have $\|P_\eps(\param)\|_{L(\calH)}\le1$ and therefore 
\begin{equation}\label{eq:project}
\|P_\eps({\wt \param})(f(\Phi(\param))-f(\Phi({\wt \param})))\|_\calH \le L\|\Phi(\param)-\Phi({\wt \param})\|_\calH.
\end{equation}
As for the sensitivity of $P_\eps(\param)$ with respect to $\param$, we consider the regularized least squares problem 
with fixed right hand side $b\in\calH$ and its normal equation
\[
M_\eps(\param)\dot \param_b(\param) = \Phi'(\param)^*b. 
\]
We differentiate
\[
\partial_\param M_\eps(\param)\dot \param_b(\param) + M_\eps(\param)\partial_\param \dot \param_b(\param) = \partial_\param \Phi'(\param)^*b
\]
and obtain
\begin{align*}
&\Phi'(\param)\partial_\param \dot \param_b(\param) 
= \Phi'(\param) M_\eps(\param)^{-1}\left(\partial_\param \Phi'(\param)^*b- \partial_\param M_\eps(\param)\dot \param_b(\param)\right)\\
&= \Phi'(\param) M_\eps(\param)^{-1}\left(\partial_\param \Phi'(\param)^*\left(b-\Phi'(\param)\dot \param_b(\param)\right) - 
\Phi'(\param)^* \partial_\param \Phi'(\param)\dot \param_b(\param)\right). 
\end{align*}
By Lemma~\ref{lem:P-bound}, we have 
$\|\Phi'(\param)M_\eps(\param)^{-1}\|_{L(\paramSpace,\calH)} \le \frac{1}{2\eps}$ and consequently
\begin{align*}
\|\Phi'(\param)\partial_\param \dot \param_b(\param)\|_\calH 
&\le \frac{\beta}{2\eps} \|b-\Phi'(\param)\dot \param_{\textcolor{black}{b}}(\param)\|_\calH + \beta\|\dot \param_b(\param)\|_\paramSpace
\le \frac{3\beta}{2\eps}\,\delta_b(\param)
\end{align*}
with
\[
\delta_b(\param)^2 =  \| \Phi'(\param)\dot \param_b(\param) - b \|_\calH^2 +\eps^2 \| \dot \param_b(\param) \|_\paramSpace^2.
\]
This implies, for 
\[
\partial_\param (P_\eps(\param)b) = \partial_\param (\Phi'(\param)\dot \param_b(\param)) 
= \partial_\param \Phi'(\param) \dot \param_b(\param) + \Phi'(\param) \partial_\param\dot \param_b(\param),
\]
the bound
\[
\|\partial_\param (P_\eps(\param)b)\|_{\calH} \le \frac{5\beta}{2\eps}\,\delta_b(\param).
\]
Using this bound for the right-hand side $b=f(\Phi(\param))$ and combining it with the Lipschitz estimate \eqref{eq:project}, we obtain
\begin{align*}
   &\|\Phi'(\param)\dot \param(\param)-\Phi'({\wt \param})\dot \param({\wt \param})\|_\calH \\*[1ex]
   &\le 
   \|(P_\eps(\param)-P_\eps({\wt \param}))f(\Phi(\param))\|_\calH + \|P_\eps({\wt \param})(f(\Phi(\param))-f(\Phi({\wt \param})))\|_\calH \\
   &\le     \frac{5\beta\bar\delta}{2\eps} \|\param-{\wt \param}\|_\paramSpace + L\|\Phi(\param)-\Phi({\wt \param})\|_\calH,
\end{align*}
which is the first of the stated bounds. The second one follows in the same way. Finally, with $u=\Phi(\param)$ and $\dot u =\Phi'(\param)\dot \param$ and analogously $\wt u$ and $\dot{\wt u}$ we have by the 
Cauchy--Schwarz inequality and \eqref{lsq-q}
\begin{align*}
    &|\delta(\param)^2-\delta(\wt \param)^2| = \big| \bigl\langle  (\dot u - f(u)) + (\dot{\wt u} - f(\wt u)), 
    (\dot u - \dot{\wt u}) - (f(u)-f(\wt u)) \bigr\rangle_\calH 
    \\
    &\qquad\qquad \qquad\qquad 
    +\eps^2 \bigl\langle \dot \param + \dot{\wt \param}, \dot \param - \dot{\wt \param} \bigr\rangle_\paramSpace \big|
    \\
    &\le \bigl(\delta(\param) + \delta(\wt \param)\bigr) \,\bigl( \| \dot u - \dot{\wt u} \|_\calH + \| f(u)-f(\wt u) \|_\calH \bigr)
    % \\
    % &\qquad\qquad \qquad\qquad     
    + \bigl(\delta(\param) + \delta(\wt \param)\bigr) \, \eps \| \dot \param - \dot{\wt \param} \|_\paramSpace,
\end{align*}
which yields 
\[
|\delta(\param)-\delta(\wt \param)| \le \| \dot u - \dot{\wt u} \|_\calH + L\| u - \wt u \|_\calH + \eps \| \dot \param - \dot{\wt \param} \|_\paramSpace,
\]
and the stated bound follows from the first two bounds.
\qed
\end{proof}

\begin{remark} As far as the inverse powers with respect to $\eps$ are concerned, the previous estimates are sharp, as the following one-dimensional example with $\cal H = \paramSpace = \R$, $\Phi(\param)=\frac12 \param^2$, and $f(u) = 1$ illustrates. In this case we have
\[
\dot \param(\param) =  \frac{\param}{\param^2 + \eps^2},\quad
d(\param) = \Phi'(\param) \dot \param(\param) - f(u) = -\frac{\eps^2}{\param^2 + \eps^2}
\]
and
\[
\partial_\param \dot \param(\param) = -\frac{\param^2}{\param^2+\eps^2} + \frac{\eps^2}{(\param^2+\eps^2)^2} = 
-\frac{d(\param)}{\param^2+\eps^2} - \frac{\param^2}{\param^2+\eps^2}.
\]
Moreover, 
\begin{align*}
|\dot u - \dot{\wt u}| &= 
\left|\frac{\param}{\param^2+\eps^2}d(\wt \param)+\frac{\wt \param}{\wt \param^2+\eps^2}d(\param)\right| \left|\param-\wt \param\right| \\
&= |\delta(\param)-\delta(\wt \param)| 
%= \frac{\eps^2|q+\wt q|}{(q^2+\eps^2)(\wt q^2+\eps^2)} \,|q-\wt q| 
%= \left|\frac{q r(\wt q)}{(q^2+\eps^2)} + \frac{\wt q r(q)}{(\wt q^2+\eps^2)}  \right|,|q-\wt q|,
\end{align*}
which implies
\begin{align*}
&|\dot u - \dot{\wt u}| = |\delta(\param)-\delta(\wt \param)| \le \frac{\bar\delta}{\eps}|\param-\wt \param|\\
&|\eps \dot \param(\param) -\eps \dot \param(\wt \param) | \le \left(\eps+\frac{\bar\delta}{\eps}\right)|\param-\wt \param|.
\end{align*}
These tight estimates exhibit the factor $\bar\delta/\eps$ in the same way as the general bounds of Proposition~\ref{prop:sensitivity}.
\end{remark}

\begin{remark}
    Using Proposition~\ref{prop:sensitivity} and Gronwall's lemma, we obtain the bound
    \[
\|u(t)-\wt u(t)\|_\calH + \eps \|\param(t)-\wt \param(t)\|_\paramSpace \le e^{\omega t} \left( \|u_0-\wt u_0\|_\calH + \eps \|\param_0-\wt \param_0\|_\paramSpace\right)
\]
with the exponent $\omega = L+ \tfrac52\beta\bar\delta/\eps^2$. For nonlinear parametrizations (where $\beta\ne 0$), which is our main interest in this paper, this is not a useful estimate for small $\eps$. For linear parametrizations ($\beta=0$), however, we even obtain
$$
\|u(t)-\wt u(t)\|_\calH  \le e^{L t}  \|u_0-\wt u_0\|_\calH 
$$
without any dependence on $\|\param_0-\wt \param_0\|_\paramSpace$. Hence, with respect to error propagation, linear and nonlinear near-singular parametrizations behave very differently.
\end{remark}

\section{Time discretization by the regularized parametric Euler method}
\label{sec:time-disc}

We study time-stepping methods in the framework of Sections~\ref{sec:regdyn} and~\ref{sec:err}, first based on the Euler method and in the next section on general Runge--Kutta methods. The unusual feature is that the differential equation for the parameters $\param$, which is the equation that is actually discretized, is ill-behaved, but nevertheless we find better behaviour of $u=\Phi(\param)$. We only assume that the vector field $f$ in the original differential equation \eqref{ivp} for $y$ is sufficiently differentiable, and that the parametrization map $\Phi$ has bounded second derivatives (but a regular parametrization is not assumed). These properties will be used to derive error bounds for the discrete approximations $u_n=\Phi(\param_n)$ at $t_n=t_0+nh$, for stepsizes $h>0$ that need to be suitably restricted in terms of the defect in the regularized least squares problem and the regularization parameter.  No error bounds are derived for the parameters $\param_n$.

\subsection{The regularized explicit Euler method}
A step of the explicit Euler method applied to the differential equation for the parameters $\param$, starting from $\param_n$ at time $t_n$ with the regularization parameter $\eps_n$, reads
\begin{equation}\label{euler-a}
\param_{n+1}=\param_n + h \dot \param_n, \qquad u_{n+1}=\Phi(\param_{n+1}),
\end{equation}
where $\dot \param_n$ is the solution of the regularized linear least squares problem
\begin{equation}\label{euler-b}
\delta_n^2 := \| \Phi'(\param_n)\dot \param_n - f(u_n)\|_\calH^2 + \eps_n^2 \| \dot \param_n \|_\paramSpace^2 \quad\text{is minimal.}
\end{equation}

\subsection{Local error bound}
We first consider the local error, that is the error after one step $u_1-y(t_1)$, where $y(t)$ is the solution of \eqref{ivp} starting from $y(t_0)=u_0=\Phi(\param_0)$. \textcolor{black}{The following results require a stepsize restriction of the form
\begin{equation}\label{euler-h}
h \delta_0 \le c \, \eps_0^2,
\end{equation}
which is mild for a small defect size $\delta_0$. Condition \eqref{euler-h}  can equally be viewed as a restriction of the regularization parameter $\eps_0$, which must not be chosen small compared to $\sqrt{h \delta_0}$. Note that a lower bound on $\eps_0$ did not occur in the time-continuous setting, but it becomes essential after time discretization.}

\begin{lemma}\label{lem:loc-err-euler}
\textcolor{black}{Under condition \eqref{euler-h},}
the local error of the regularized Euler method starting from $y(t_0)=u_0$ is bounded by
\begin{equation}\label{loc-err-euler}
\| u_1-y(t_1) \|_\calH \le c_1 h\delta_0 + c_2 h^2
\end{equation}
with $c_1=1+\frac12c\beta$, where $\beta$ is a bound of the second derivative of $\Phi$ in a neighbourhood of $\param_0$,
and $c_2=\frac12\max_{t_0\le t \le t_1} \| \ddot y(t) \|_\calH$.
\end{lemma}

\begin{proof} We have by Taylor expansion
\begin{align*}
y(t_1)- y(t_0) &= h \dot y(t_0) + \int_{t_0}^{t_1} (t_1-t) \, \ddot y(t) \, dt
\\
&=  h f(u_0) + O(h^2)
\end{align*}
and we have
\begin{align*}
u_1 - u_0 &= \Phi(\param_1) - \Phi(\param_0)
\\
&= \Phi'(\param_0)(\param_1\!-\!\param_0) + \!\!\int_0^1 \! (1\!-\!\xi) \Phi''(\param_0\!+\!\xi(\param_1\!-\!\param_0)) [\param_1\!-\!\param_0,\param_1\!-\!\param_0] \,d\xi
\\
%&= \Phi'(\param_0)(\param_1-\param_0) + O(\| \param_1-\param_0 \|_\paramSpace^2)
%\\
&= \Phi'(\param_0) h \dot \param_0 + O(\| h\dot \param_0 \|_\paramSpace^2)
\\
&= h f(u_0) + O(h\delta_0) + O\biggl(\Bigl(\frac{h\delta_0}{\eps_0}\Bigr)^2\biggr),
\end{align*}
where the last line results from the definition of $\delta_0$, which is an upper bound of both 
$\| \Phi'(\param_0)  \dot \param_0 - f(u_0) \|_\calH$ and $\eps_0 \|\dot \param_0\|_\paramSpace$. This implies $\| h\dot \param_0 \|_\paramSpace \le h\delta_0/\eps_0$.
Under the stepsize restriction $h \delta_0 \le c \eps_0^2$, the last term is also $O(h\delta_0)$. Subtracting the two formulas and tracing the constants in the $O(\cdot)$ terms yields the stated result.
\qed
\end{proof}

\begin{remark}\label{rem-step-size-1} The result of Lemma~\ref{lem:loc-err-euler} holds without any fixed bound on $\dot \param_0$. If, however, $\| \dot \param_0\|_\paramSpace \le \gamma$ with a moderate constant $\gamma$, which is satisfied if 
\begin{equation} \label{euler-eps}
\delta_0\le \gamma \eps_0,
\end{equation}
then the above proof shows that the error bound \eqref{loc-err-euler} is satisfied with $c_1=1$ and $c_2=\frac12\max_{t_0\le t \le t_1} \| \ddot y(t) \|_\calH + \beta \gamma^2$ with the same $\beta$ as in the lemma. This clearly indicates that while the stepsize restriction \eqref{euler-h} is sufficient for the stated error bound, it is not always necessary. \textcolor{black}{However, both \eqref{euler-h} and \eqref{euler-eps} indicate that the regularization parameter must not be chosen small with respect to the stepsize and the defect, as is also observed in numerical experiments.}
\end{remark}

\begin{remark}\label{a-posteriori}
A local error estimate of a similar structure, without any restriction on the step size $h$ or the defect $\delta_0$, is readily available by keeping an additional a posteriori error term. Comparing the explicit Euler approximation 
$y_1=y_0+hf(y_0)$ as an intermediate term with $u_1=\Phi(\param_1)$ yields
\begin{align}
&\| u_1-y(t_1) \|_\calH \le \| u_1 - y_1\|_\calH + \| y_1 - y(t_1) \|_\calH
\nonumber
\\
&\le \| u_1 - y_0 - h\Phi'(\param_0)\dot{\param}_0 \|_\calH + h \| \Phi'(\param_0)\dot{\param}_0 - f(u_0) \|_\calH + \| y_1 - y(t_1) \|_\calH
\nonumber
\\
&\le \| u_1 - y_0 - h\Phi'(\param_0)\dot{\param}_0 \|_\calH + h \delta_0 + O(h^2).
\label{loc-err-euler-2}
\end{align}
The first term in the last line is computable with the quantities derived during the time integration of the scheme and comes at the cost of evaluating an additional $\mathcal H$-norm.
Under the step size restriction \eqref{euler-h}, or in the case of Remark~\ref{rem-step-size-1}, the term is of the optimal order $O(h\delta_0+ h^2)$.
\end{remark}

\begin{remark}
 \textcolor{black}{
It is a natural question to ask if the stepsize restriction \eqref{euler-h} can be avoided by choosing instead of the explicit Euler method an implicit method for stiff differential equations such as the implicit Euler method. Theoretical and numerical results in \cite{LN25}
indicate that the same condition \eqref{euler-h} or \eqref{euler-eps} is also required for implicit methods. 
}   
\end{remark}

\subsection{Global error bound (using stable error propagation by the exact flow)}
Using the standard argument of Lady Windermere's fan with error propagation by the exact flow, see the book by Hairer, N\o rsett \& Wanner \cite[II.3]{HNW}, we obtain the following global error bound from Lemma~\ref{lem:loc-err-euler} and condition \eqref{Lip}.

\begin{theorem}\label{prop:glob-err-euler}
Under condition \eqref{Lip} and the stepsize restriction 
$$
h \delta_n \le c \,\eps_n^2, \qquad 0\le n \le N,
$$
the error of the regularized Euler method \eqref{euler-a}--\eqref{euler-b} with initial value $y_0=u_0=\Phi(\param_0)$ is bounded, for $t_n=nh\le t_N \le \bar t$, by 
$$ 
\| u_n-y(t_n)\|_\calH =O(\delta+h) \quad\text{with}\quad \delta=\max_n \delta_n,
$$
or more precisely (compare with Proposition~\ref{prop:delta}),
$$
\| u_n - y(t_n) \|_\calH \le h \sum_{j=0}^{n-1} e^{\ell(t_{n-1}-t_j)} \, \bigl(c_1\delta_j + c_2  h ),
$$
where $c_1=1+\frac12 c\beta$ with $\beta$ a bound of the second derivative of $\Phi$ in a neighbourhood of  the solution,
and $c_2$ is a bound of $f'f$ in a neighbourhood of the solution $y(t)$, i.e. a bound of second derivatives of solutions $\widetilde y$ of the differential equation $\dot{\widetilde y}=f(\widetilde y)$ with initial values in a neighbourhood of the solution trajectory $\{ y(t): 0 \le t \le \bar t \}$. 
\end{theorem}

\subsection{Error propagation by the numerical method}
The Euler method applied to \eqref{ivp} is stable in the sense that the numerical results $y_1$ and $\wt y_1$ obtained from starting values $y_0$ and $\wt y_0$, respectively, differ by
$$
\| y_1 - \wt y_1 \|_\calH \le (1+hL) \| y_0 - \wt y_0 \|_\calH,
$$
where $L$ is again the Lipschitz constant of $f$. For the regularized Euler method we only obtain, similarly to the proof of Lemma~\ref{lem:loc-err-euler}, a discrete analogue of \eqref{u-diff}:
\begin{equation}\label{u-diff-Euler}
    \| u_1 - \wt u_1 \|_\calH \le (1+hL) \| u_0 - \wt u_0 \|_\calH + h\delta_0 + h\wt\delta_0.
\end{equation}
This is not good enough to be used in Lady Windermere's fan with error propagation by the numerical method, which is the most common way to prove error bounds for numerical methods for nonstiff differential equations; see \cite{HNW}.

\bch
We can, however, use Lady Windermere's fan with error propagation by the abstract unparametrized numerical method (here the Euler method), both to compare the results of the parametrized and unparametrized method and to bound the global error of the unparametrized method, which combined give a global error bound for the regularized parametrized method. In the present case this procedure would yield essentially the same error bound as in Theorem~\ref{prop:glob-err-euler}.
We will use this procedure in Section~\ref{sec:psi}, where error propagation by the exact flow cannot be used because it would require higher regularity of the numerical solution that is not known.
\ech

\subsection{An {\it a priori} error bound for the regularized explicit Euler method}
The bound in Theorem~\ref{prop:glob-err-euler} is partly {\it a priori} (in its dependence on $\ell$, $c_1$ and $c_2$) and partly {\it a posteriori} (in its dependence on the defect sizes $\delta_j$). In terms of the method-independent \textcolor{black}{-- and in this sense {\it a priori} --} defect sizes $\bar\delta_\rho(t)$ $\rho$-near to the exact solution as defined in
\eqref{delta-rho}, we have the following discrete analogue of Proposition~\ref{prop:delta-star}.

\begin{proposition}
    \label{prop:delta-star-euler}
    In the situation of Theorem~\ref{prop:glob-err-euler} we have the \emph{a priori} error bound, for $0\le t_n=nh \le \bar t$,
    $$
\| u_n-y(t_n) \|_\calH \le C_1 h \sum_{j=1}^n \bar\delta_\rho(t_j) + C_2 h,
$$
where $C_1$ and $C_2$ depend on $L\bar t$ but are independent of $h$ and $\eps$. This bound is valid as long as the right-hand side does not exceed $\rho$.
\end{proposition}

\begin{proof}
We still have the bound of Lemma~\ref{lem:delta-bound},
$$
\delta_n \le \bar\delta_\rho(t_n) + \| f(u_n) - f(y(t_n)) \|_\calH.
$$
Inserting this bound in the local error bound of Theorem~\ref{prop:glob-err-euler} and using the Lipschitz condition on $f$ yields
$$
\| u_n-y(t_n) \|_\calH \le c_1 h \sum_{j=1}^n \bar\delta_\rho(t_j) + c_1 h \sum_{j=1}^n L\,\| u_j - y(t_j) \|_\calH + c_2 h,
$$
and the discrete Gronwall inequality yields the result.
\qed
\end{proof}

 \section{Regularized parametric explicit Runge--Kutta methods} \label{sec:rk}
A step of an explicit Runge--Kutta method of order $p$ with coefficients $a_{ij}$ and $b_j$ applied to the differential equation for the parameters $\param$ reads as follows. Starting from $\param_n$ at time $t_n$ with the regularization parameter $\eps_n$, we first compute consecutively the internal stages (for $i=1,\dots,s$)
\begin{equation}\label{rk-a}
\param_{n,i}=\param_n + h \sum_{j=1}^{i-1} a_{ij} \dot \param_{n,j}, \qquad u_{n,i}=\Phi(\param_{n,i}),
\end{equation}
and $\dot \param_{n,i}$ as the solution of the regularized linear least squares problem
\begin{equation}\label{rk-b}
\delta_{n,i}^2 := \| \Phi'(\param_{n,i})\dot \param_{n,i} - f(u_{n,i})\|_\calH^2 + \eps_n^2 \| \dot \param_{n,i} \|_\paramSpace^2 \quad\text{is minimal,}
\end{equation}
\textcolor{black}{which after reformulation in damped form, see \eqref{reg-damp}, can be solved by a large class of numerical methods.}
Finally, the new value is computed as
\begin{equation}\label{rk-c}
\param_{n+1}=\param_n + h \sum_{j=1}^s b_j \dot \param_{n,j}, \qquad u_{n+1}=\Phi(\param_{n+1}).
\end{equation}
We first bound the local error.
\begin{lemma}\label{lem:loc-err-rk}
Let $\delta_0=\max_i \delta_{0,i}$. Under the stepsize restriction 
\begin{equation}\label{rk-h}
h \delta_0 \le c \,\eps_0^2,
\end{equation}
the local error of the regularized $p$-th order Runge--Kutta method starting from $y(t_0)=u_0=\Phi(\param_0)$ is bounded by
\begin{equation}\label{loc-err-rk}
\| u_1-y(t_1) \|_\calH \le c_1 h\delta_0 + c_2 h^{p+1},
\end{equation}
where $c_1$ is a constant times $1+\frac12c\beta$ with the bound $\beta$ of the second derivative of $\Phi$ in a neighbourhood of $\param_0$,
and $c_2 h^{p+1}$ is the bound for the local error of the $p$-th order Runge--Kutta method applied to \eqref{ivp}.
\end{lemma}

\begin{proof}
We note that \eqref{rk-b} implies $\|\eps_0\dot \param_{0,j}\|_\paramSpace \le \delta_0$ and hence 
$\| h \dot \param_{0,j}\|_\paramSpace \le h\delta_0/\eps_0$. Moreover, \eqref{rk-b} also implies
$\| \Phi'(\param_{0,j})\dot \param_{0,j}- f(u_{0,j})\|_\calH \le \delta_0$.
This yields 
\begin{align*}
&u_{0,i} - u_0 = \Phi(\param_{0,i}) - \Phi(\param_0)
%\\
%&
= \Phi'(\param_0) h \sum_{j=1}^{i-1} a_{ij} \dot \param_{0,j} + O\biggl(\Bigl(\frac{h\delta_0}{\eps_0}\Bigr)^2\biggr)
\\
&= h \sum_{j=1}^{i-1} a_{ij} \Phi'(\param_{0,j} )\dot \param_{0,j} 
-  \sum_{j=1}^{i-1} a_{ij} \bigl( \Phi'(\param_{0,j} ) - \Phi'(\param_{0} )\bigr)h\dot \param_{0,j}+
O\biggl(\Bigl(\frac{h\delta_0}{\eps_0}\Bigr)^2\biggr)
\\
&= h \sum_{j=1}^{i-1} a_{ij} f(u_{0,j}) + O(h\delta_0) + O\biggl(\Bigl(\frac{h\delta_0}{\eps_0}\Bigr)^2\biggr)
+ O\biggl(\Bigl(\frac{h\delta_0}{\eps_0}\Bigr)^2\biggr)
\end{align*}
and in the same way
$$
u_1-u_0 = h \sum_{j=1}^{s} b_j f(u_{0,j}) + O(h\delta_0) + O\biggl(\Bigl(\frac{h\delta_0}{\eps_0}\Bigr)^2\biggr).
$$
Apart from the $O(\cdot)$ terms, these formulae are those that define the result $y_1$ of one step of the Runge--Kutta method applied to \eqref{ivp}. So we obtain, using also the Lipschitz continuity of $f$,
$$
u_1 - y_1 = O(h\delta_0) + O\biggl(\Bigl(\frac{h\delta_0}{\eps_0}\Bigr)^2\biggr).
$$
Since the method is of order $p$ and $f$ is sufficiently differentiable, we have
$$
y_1 - y(t_1) = O(h^{p+1}).
$$
Noting that under the stepsize restriction \eqref{rk-h} we have $({h\delta_0}/{\eps_0})^2 \le ch\delta_0$,
%$$
%\Bigl(\frac{h\delta_0}{\eps_0}\Bigr)^2 \le ch\delta_0,
%$$
the error bound \eqref{loc-err-rk} follows.
\qed
\end{proof}

As before, using the local error bound in Lady Windermere's fan with error propagation by the exact solutions, we obtain the following global error bound for the $p$-th order Runge method. We formulate the result for variable stepsizes $h_n$, so that $t_{n+1}=t_n+h_n$ and $t_N=\bar t$.

\begin{theorem}\label{prop:glob-err-rk}
Under condition \eqref{Lip} and the stepsize restriction 
$$
h_n \delta_{n,i} \le c \,\eps_n^2, \qquad 0\le n \le N,\ \  i=1,\dots,s,
$$
the error of the regularized $p$-th order Runge--Kutta  method \eqref{rk-a}--\eqref{rk-c} with initial value $y_0=u_0=\Phi(\param_0)$ is bounded, for $t_n\le \bar t$, by 
$$
\| u_n-y(t_n)\|_\calH =O(\delta+h^p)
$$
with $\delta=\max_{n,i} \delta_{n,i}$ and $h=\max_n h_n$. The constants symbolized by the O-notation are independent of $\,\delta$ and the regularization parameters $\eps_n$ (under the given stepsize restriction), and of the stepsize sequence $(h_k)$ and $n$ with $t_n\le \bar t$.
\end{theorem}

\section{Choice of the regularization parameter and the stepsize}
\label{sec:eps-h}

The algorithm below chooses the regularization parameter $\eps_n$ in the $n$th time step as large as possible such that the defect
$\delta_n$ is within a given factor of the defect attained for tiny $\eps$ or within a prescribed tolerance. The stepsize $h_n$ is chosen such that the critical quadratic error terms are of size $h_n\delta_n$.

For arbitrary $\eps>0$, in the $n$th time step we let $\dot \param_n(\eps)$ be the solution of the regularized linear least squares problem with regularization parameter $\eps$ such that
$$
\delta_n(\eps)^2 := \| \Phi'(\param_n)\dot \param_n(\eps) -f(u_n) \|_\calH^2 + \eps^2 \| \dot \param_n(\eps) \|_\paramSpace^2 \quad\text{is minimal.}
$$
Let $\eps_n^0$ and $h_n^0$ be initializations of the regularization parameter and the stepsize, respectively. These might be the values from the previous time step. Let $\eps_\star$ be a tiny reference parameter such that $\delta_n(\eps_\star)$ can still be computed reliably (just so). Let $\delta_{\min}>0$ be a given threshold.

\subsection{Choice of the regularization parameter $\eps_n$}

%Initialize $\eps_n=\eps_n^0$ and $h_n=h_n^0$.
%\\
Compute $\delta_n(\eps_\star)$ and set $\delta_n^{\rm tol} = \max(17 \,\delta_n(\eps_\star),\delta_{\min})$ as the target defect size\footnote{The factor 17, chosen in honour of Gau\ss,  can be replaced by a different factor {\it ad libitum}.}.
We aim at having $\delta_n(\eps_n) \approx \delta_n^{\rm tol}$. We use 1 or 2 Newton iterations for the equation $\delta_n(\eps_n)^2 - (\delta_n^{\rm tol})^2=0$ (except in the very first step where a good starting value is not available and more iterations might be needed). In view of Lemma~\ref{lem:reg-lsq}, which shows that this is a monotonically increasing concave smooth function with the derivative $d\delta_n^2/d\eps^2=\| \dot \param_n(\eps)\|^2_\paramSpace$, we have the Newton iteration
\begin{align*}
   \left( \varepsilon_n^{k+1} \right)^2 =    \left( \varepsilon_n^{k} \right)^2 - \frac{\delta_n(\varepsilon_n^{k})^2 -\left(\delta_n^{\text{tol}}\right)^2}{\|\dot \param_n(\eps_n^k)\|_\paramSpace^2}.
\end{align*}
Alternatively, we can proceed in a Levenberg--Marquardt style, cf.~\cite{more2006levenberg}:\\ 
Initialize $\eps_n=\eps_n^0$.\\
If $\delta_n(\eps_n) \le \delta_n^{\rm tol}$
\\
$\phantom{.}\quad$ then while $\delta_n(3\eps_n) \le \delta_n^{\rm tol}$ set $\eps_n:=3\eps_n$
\\
$\phantom{.}\quad$ else while $\delta_n(\eps_n)>\delta_n^{\rm tol}$ set $\eps_n:=\eps_n/3$.

\medskip\noindent
In the following we set $\dot \param_n=\dot \param_n(\eps_n)$ and $\delta_n=\delta_n(\eps_n)$.

\subsection{Choice of the stepsize $h_n$}
The stepsize $h_n$ is chosen such that
$\| \Phi''(\param_n)[h_n\dot \param_n,h_n\dot \param_n] \|_\calH \approx h_n \delta_n$.
This is satisfied for the choice
$$
h_n = \frac{h_n^0\delta_n }{\|\Phi'(\param_n+h_n^0\dot \param_n)\dot \param_n -\Phi'(\param_n)\dot \param_n \|_\calH}.
$$

\subsection{Regularized Runge--Kutta step with $\eps_n$ and $h_n$}
We compute $\param_{n+1}$ and $u_{n+1}=\Phi(\param_{n+1})$ by a regularized Runge--Kutta step (see Section~\ref{sec:rk}) with the proposed regularization parameter $\eps_n$ and the proposed stepsize $h_n$. We can use an embedded pair of Runge--Kutta methods that gives a local error estimate that should not substantially exceed $h_n\delta_n$ (else the step is rejected and repeated with a reduced stepsize).

\bigskip
\textcolor{black}{
While we do not present numerical experiments with the above algorithm in this paper, we refer to \cite{LN25}, where the results of numerical experiments are reported for a similar algorithm that adaptively determines the regularization parameter $\eps_n$ for a fixed stepsize $h$.
}

\section{Conserved quantities}
\label{sec:con}

\subsection{Conserved quantities and regularized dynamical approximation}
Let $g=(g_1,\dots,g_m)^\top :\calH\to \R^m$ be an $m$-vector of real-valued functions that are conserved along the flow of the differential equation \eqref{ivp}:
$$
g(y(t))= g(y(0)) \qquad \text{for all $t$}
$$
for every choice of initial value $y(0)\in \calH$. Differentiating this equation w.r.t.~$t$, this is seen to be equivalent to
$$
G(y(t))\dot y(t)=0,
$$
where $G=g'=\partial g/\partial y$, and hence to
$$
G(y)f(y) = 0  \qquad \text{for all $y$}.
$$
However, $g$ is no longer a conserved quantity for the regularized dynamical approximation \eqref{reg-lsq}, not even for linear $g$. We use the notation 
$$
A=\Phi',\quad M_\eps=A^\top  \! A + \eps^2 I,\quad
P_\eps =A M_\eps^{-1} \! A^\top  
$$
where we omitted the argument $\param$ for all appearing matrices.
For the regularized least squares problem \eqref{reg-lsq} we have the normal equations (now omitting the argument $t$)
\begin{equation}\label{q-ode}
M_\eps(\param)\dot \param = A(\param)^\top  f(\Phi(\param))
\end{equation}
and for $u=\Phi(\param)$, the time derivative $\dot u = \Phi'(\param)\dot \param$ becomes
\begin{equation}\label{u-eq-d}
\dot u = P_\eps(\param)f(u) = f(u) + d \quad\text{ with }\quad
d=-(I-P_\eps(\param))f(u),
\end{equation}
where the defect $d$ is bounded by $\| d \|_\calH \le \delta$ in view of \eqref{reg-lsq}.
As $G(u)f(u)=0$, we obtain
$$
\frac{d}{dt}\,g(u)= G(u)\dot u = G(u)d,
$$
which in general is different from zero, so that $g(u(t))$ is not conserved. We note, however, the bound
$$
|g(u(t))-g(u(0))| \le K \int_0^t \delta(t)\, dt, 
$$
where $K$ is an upper bound of the norm of $G(u)$ along the trajectory $u(\cdot)$.

\subsection{Enforcing conservation}
We can enforce conservation of $g$ along the approximation $u(t)=\Phi(\param(t))$ by adding the condition $G(u(t))\dot u(t)=0$ as a constraint, i.e. (omitting the argument $t$)
$$
G(\Phi(\param))A(\param)\dot \param = 0,
$$
and we minimize in \eqref{reg-lsq} under this constraint.
With the notation
$$
C(\param):=G(\Phi(\param))A(\param)=g'(u)\Phi'(\param),
$$
we obtain instead of \eqref{q-ode} the constrained system with a Lagrange multiplier $\lambda(t)\in\R^m$,
\begin{equation} \label{constrained}
    \begin{aligned}
     M_\eps(\param)&\dot \param +C(\param)^\top  \lambda = A(\param)^\top  f(\Phi(\param))\\
        C(\param) &\dot \param        \hskip 16mm             = 0.
    \end{aligned}
\end{equation}
Inserting $\dot \param$ from the first equation into the second equation, we find $\lambda$ from
the equation (omitting the argument $\param$ or $u$ of the matrices)
$$
CM_\eps^{-1}C^\top  \lambda = C M_\eps^{-1} A^\top  f(u)
$$
or equivalently, since $C=GA$ and $P_\eps=AM_\eps^{-1}A^\top$ imply $CM_\eps^{-1}C^\top=G P_\eps G^\top$ and $C M_\eps^{-1} A^\top= G P_\eps$, 
and since $(P_\eps - I) f(u)=d$ and $G(u)f(u)=0$, we find
\begin{equation}\label{lambda}
\bigl(G(u)P_\eps(\param) G(u)^\top\bigr)  \lambda = G(u)d.
\end{equation}
The symmetric positive semi-definite matrix $P_\eps=A M_\eps^{-1} A^\top $ has the eigenvalues $\lambda_i =\sigma_i^2 / (\sigma_i^2 + \eps^2)$ and 0, where
$\sigma_i$ are the singular values of $A$. Eigenvalues are very small if they correspond to very small singular values $\sigma_i\ll\eps$ of $A$, but are larger than $\tfrac12$ for $\sigma_i\ge \eps$. To understand under which condition the symmetric positive semi-definite matrix $GP_\eps G^\top $ has a moderately bounded inverse, let $\Lambda$ be the diagonal matrix of eigenvalues of $P_\eps$ and $U$ the orthogonal  matrix of eigenvectors, and for $\xi>0$  let $U_\xi$ be the matrix composed of those eigenvectors of $P_\eps$ that correspond to the eigenvalues  $\lambda_i \ge \xi$. 
If the smallest singular value of $GU_\xi$ equals $\rho >0$, then 
\begin{equation}\label{GPG-inv}
\| (GP_\eps G^\top )^{-1} \|_2 \le \frac 1{\xi \rho^2},
\end{equation}
%as follows from 
because $v^\top  GU \Lambda (GU)^\top  v \ge v^\top  GU_\xi \Lambda_\xi (GU_\xi)^\top  v \ge \xi \|(GU_\xi)^\top  v\|_2^2 \ge \xi \rho^2 \|v\|_2^2$ for all $v\in \R^m$. 
Here, $\Lambda_\xi$ is the diagonal matrix of those eigenvalues of $P_\eps$ that are larger than~$\xi$. We remark that in the case of just one conserved quantity ($m=1$), the inverse of $GP_\eps G^\top  \in \R$ is moderately bounded if $\nabla g$ is not near-orthogonal to all those singular vectors of $A$ that correspond to singular values $\sigma_i\ge \eps$. This appears to be a very mild requirement.

Inserting $\lambda$ from \eqref{lambda} into the first equation of \eqref{constrained} and using the definition of the defect $d$  yields $\dot u = A\dot \param$ as
$$
\dot u + P_\eps G^\top  (GP_\eps G^\top )^{-1}G d = f(u) +d.
$$
With the projection onto the null-space of $G(u)$ for $u=\Phi(\param)$ that is given by
$$
\Pi(\param) = I - P_\eps(\param) G(u)^\top  (G(u)P_\eps(\param) G(u)^\top )^{-1}G(u),
$$
we thus obtain for $u(t)=\Phi(\param(t))$, instead of \eqref{u-eq-d}, the differential equation with the projected defect,
\begin{equation}\label{u-eq-Pd}
\dot u = f(u) + \Pi(\param)d.
\end{equation}

\subsection{Constrained regularized Euler method}
We now ensure the condition $g(u_{n+1})=g(u_n)$ by adding it as a constraint to the regularized least squares problem \eqref{euler-b}. 
A step of the constrained regularized Euler method, starting from $\param_n$ at time $t_n$ with the regularization parameter $\eps_n$, reads
\begin{equation}\label{euler-a-con}
\param_{n+1}=\param_n + h \dot \param_n, \qquad u_{n+1}=\Phi(\param_{n+1}),
\end{equation}
where $\dot \param_n$ is the solution of the constrained regularized linear least squares problem
\begin{equation}\label{euler-b-con}
\begin{aligned}
&\wh\delta_n^{\,2} := \| \Phi'(\param_n)\dot \param_n - f(u_n)\|_\calH^2 + \eps_n^2 \| \dot \param_n \|_\paramSpace^2 \quad\text{is minimal}
\\
&\text{subject to }\ g(u_{n+1})=g(u_n).
\end{aligned}
\end{equation}
Note that while $\delta_n$ of \eqref{euler-b} depends only on $\param_n$, the defect size $\wh\delta_n\ge \delta_n$ depends also on the stepsize $h$.
With the notation of the previous subsections, a step of the unconstrained regularized Euler method of Section~\ref{sec:time-disc} reads (with $\eps=\eps_n$)
$$
M_\eps(\param_n ) \dot \param_n  = A(\param_n )^\top  f(u_n )
$$
together with $\wt \param_{n+1} =\param_n +h\dot \param_n $ and $\wt u_{n+1} =\Phi(\param_{n+1} )$. The minimality condition for the constrained problem \eqref{euler-b-con} now determines (a different) $\dot \param_n $ together with the Lagrange multiplier $\lambda_{n+1} $ from the nonlinear system of equations
\begin{equation}\label{euler-con}
\begin{aligned}
&M_\eps(\param_n ) \dot \param_n  + C(\param_n )^\top  \lambda_{n+1}  = A(\param_n )^\top  f(u_n )\\
&g\bigl(\Phi(\param_n +h\dot \param_n )\bigr) = g(u_n ).
\end{aligned}
\end{equation}
We then set $\param_{n+1} =\param_n +h\dot \param_n $ and $u_{n+1} =\Phi(\param_{n+1} )$.
Inserting $\dot \param_n $ from the first equation into the second equation, we get a nonlinear equation for $\lambda_{n+1} $: with $\wt \param_{n+1}  = \param_n + h M_\eps(\param_n )^{-1} A(\param_n )^\top  f(u_n )$ (which is the result of the unconstrained Euler method) we have
$$
g\bigl(\Phi(\wt \param_{n+1}  - h M_\eps(\param_n )^{-1} C(\param_n )^\top  \lambda_{n+1}  )\bigr)-g(u_n )=0.
$$
A modified Newton method applied to this equation determines the $(k+1)$st iterate 
$\lambda_{n+1} ^{(k+1)}= \lambda_{n+1} ^{(k)} + \Delta \lambda_{n+1} ^{(k)}$ by solving the linear system
\begin{equation} \label{mod-newton}
\begin{aligned}
&-h C(\param_n ) M_\eps(\param_n )^{-1} C(\param_n )^\top  \Delta \lambda_{n+1} ^{(k)} 
\\
&\qquad\qquad = 
- g\bigl(\Phi(\wt \param_{n+1}  - h M_\eps(\param_n )^{-1} C(\param_n )^\top  \lambda_{n+1} ^{(k)} )\bigr)+g(u_n ).
\end{aligned}
\end{equation}
The starting value is chosen as $\lambda_{n+1} ^{(0)}=0$.
The matrix $C M_\eps^{-1} C^\top =G P_\eps G^\top$ is the same as in \eqref{lambda}. It is assumed to be invertible, see the bound~\eqref{GPG-inv} of the inverse.

\begin{theorem} \label{lem:newton}
If the matrix $G(u_n)P_\eps(\param_n)G(u_n)^\top$ has a moderately bounded inverse, then
the modified Newton iteration \eqref{mod-newton} with starting value ${\lambda_{n+1}^{(0)}=0}$ converges under the stepsize restriction
$$
h(\delta_n+h) \le c\eps
$$
with a sufficiently small $c$ that is independent of $h$, $\eps$ and $\delta$. Moreover, 
$$
\lambda_{n+1}= O(\delta_n+h).
$$
\end{theorem}

\begin{proof}
The modified Newton iteration is a fixed-point iteration for the map (omitting the argument $u_n$  of $G$ and $\param_n$ of $A,M_\eps,P_\eps$ and letting $z=h\lambda$)
$$
\varphi(z) = z -   (GP_\eps G^\top)^{-1}  \Bigl( 
g\bigl( \Phi( \wt \param_{n+1} -  M_\eps^{-1} A^\top G^\top z \bigr) - g(u_n) \Bigr).
$$
Using the $O(h(\delta_n+h))$ local error bound of the regularized Euler method as given in Lemma~\ref{lem:loc-err-euler} and the conservation of $g$ by the exact flow from $t_n$ to $t_{n+1}$,
we obtain for $\wt u_{n+1}=\Phi(\wt \param_{n+1})$ that $g(\wt u_{n+1})-g(u_n)=O(h(\delta_n+h))$ and hence $z^{(1)}= O(h(\delta_n+h))$ for the starting value $z^{(0)}=0$.
Using that $\| M_\eps^{-1} A^\top \| \le 1/(2\eps)$, we find that in a ball of radius $O(h(\delta_n+h))$ centered at $0$
we have
$$
\varphi'(z) = O\bigl(h(\delta_n+h)/\eps\bigr),
$$
which is strictly smaller than 1 under the given stepsize restriction.
The stated result then follows with the Banach fixed-point theorem.
\qed
\end{proof}

\subsection{Error analysis}
The local error has a bound similar to Lemma~\ref{lem:loc-err-euler} with the only difference that the constants now also depend on a bound of the inverse of the matrix in \eqref{lambda} and on bounds of derivatives of $g$. Note that the following local error bound is in terms of the defect size $\delta_0$
of the {\em unconstrained} regularized Euler method, as in Section~\ref{sec:time-disc}, not just of the larger $\wh\delta_0$ of the constrained method \eqref{euler-b-con}.

\begin{lemma}\label{lem:loc-err-euler-con}
Assume that the matrix $G(u_0) P_\eps(\param_0) G(u_0)^\top$ (with $\eps=\eps_0$) has an inverse bounded by 
$
\bigl\| \bigl(G(u_0) P_\eps(\param_0) G(u_0)^\top\bigr)^{-1} \bigr\|_2 \cdot \| G(u_0) \|_2^2 \le \gamma.
$
Under the stepsize restriction (cf. \eqref{euler-h})
\begin{equation}\label{euler-h-con}
h \wh \delta_0 \le c \eps_0^2
\end{equation}
with a sufficiently small $c$ (inversely proportional to $\gamma$),
we have 
$$
\wh\delta_0 \le \wh c \,\delta_0,
$$
where $\wh c$ is proportional to $\gamma$ but independent of $h$ and~$\eps_0$.
The local error of the regularized Euler method starting from $y(t_0)=u_0=\Phi(\param_0)$ is then bounded by
\begin{equation}\label{loc-err-euler-con}
\| u_{1} -y(t_1) \|_\calH \le \wh c_1 h\delta_0 + \wh c_2 h^2,
\end{equation}
where $\wh c_2$ equals $c_2$ of Lemma~\ref{lem:loc-err-euler} and $\wh c_1$  is proportional to $\gamma$, depends on a bound of the second derivative of $\Phi$ in a neighbourhood of $\param_0$ and
 on a bound of the second derivative of $g$ in a neighbourhood of $u_0$.
\end{lemma}

\begin{proof}
    As in the proof of Lemma~\ref{lem:loc-err-euler}, we write $y(t_1)-y(t_0)=hf(u_0)+ O(h^2)$,
    and we have again
    $$
    u_1-u_0 = A(\param_0) h\dot \param_0 + O(h^2\|\dot \param_0\|_\paramSpace^2),
    $$
    and further
    $$
    0= g(u_1)-g(u_0) = G(u_0) (u_1-u_0) + O(h^2\|\dot \param_0\|_\paramSpace^2).
    $$
    From the first equation of \eqref{euler-con} we have
    $$
    A(\param_0)\dot \param_0 + P_\eps(\param_0) G(u_0)^\top \lambda_1 = f(u_0) + d_0,
    $$
    where $d_0 = -(I-P_\eps(\param_0))f(u_0)$ is the defect of the {\em unconstrained} regularized Euler method, which is bounded by $\delta_0$.
    % where  $\|d_0- P_\eps(\param_0) G(u_0)^\top \lambda_1\|_\calH = \|A(\param_0)\dot \param_0 -  f(u_0) \|_\calH \le \delta_0$.
    From the constraint, using $G(u_0)f(u_0)=0$, we thus find
    \begin{align*}
    0 &= g(u_1)-g(u_0)= G(u_0)A(\param_0) h\dot \param_0 + O(h^2\|\dot \param_0\|_\paramSpace^2) 
    \\
    &= - h G(u_0)P_\eps(\param_0) G(u_0)^\top \lambda_1 + h G(u_0) d_0 + O(h^2\|\dot \param_0\|_\paramSpace^2),
    \end{align*}
    which yields $\lambda_1=O(\delta_0) + O(h\|\dot \param_0\|_\paramSpace^2)$ and,
    with the projection $\Pi(\param)$ appearing in \eqref{u-eq-Pd},
    $$
    u_1-u_0= h f(u_0) + h \Pi(\param_0) d_0 + O(h^2\|\dot \param_0\|_\paramSpace^2).
    $$
    Hence we obtain the error bound
    \begin{equation}\label{err-1-con}
    u_1-y(t_1)= O(h\delta_0) + O(h^2\|\dot \param_0\|_\paramSpace^2) + O(h^2).
    \end{equation}
    It remains to show that the second term on the right-hand side is also $O(h\delta_0)$ under the stepsize restriction \eqref{euler-h-con}. So far we only know from \eqref{euler-b-con} that $\eps \|\dot \param_0\|_\paramSpace \le \wh\delta_0$ and $\delta_0 \le \wh\delta_0$.
    From the first equation of \eqref{euler-con} we obtain
    $$
    \eps\dot \param_0 + \eps M_\eps(\param_0)^{-1} A(\param_0)^\top G(u_0)^\top \lambda_1 = \eps M_\eps(\param_0)^{-1} A(\param_0)^\top f(u_0) .
    $$
    We estimate
    $$
    \| \eps M_\eps(\param_0)^{-1} A(\param_0)^\top G(u_0)^\top \lambda_1 \|_\paramSpace \le \| G(u_0)^\top \lambda_1  \|_\calH
    = O(\delta_0) + O(h\|\dot \param_0\|_\paramSpace^2).
    $$
    We write $f(u_0)= A(\param_0)\dot \param_0^{\rm uncon}- d_0$, where $\dot \param_0^{\rm uncon}$ is the derivative approximation in the unconstrained regularized Euler method, which is bounded by 
    $\eps \| \dot \param_0^{\rm uncon} \|_\paramSpace \le \delta_0$.
    Using further that $\|M_\eps(\param_0)^{-1} A(\param_0)^\top A(\param_0)\|\le 1$,
    this yields the bound  
    $$
    \eps \|\dot \param_0\|_\paramSpace \le \| G(u_0)^\top \lambda_1  \|_\calH + \eps \| \dot \param_0^{\rm uncon} \|_\paramSpace + \|d_0\|_\calH = O(\delta_0) + O(h\|\dot \param_0\|_\paramSpace^2).
    $$
    Under the stepsize restriction \eqref{euler-h-con} we have 
    $$
    h\|\dot \param_0\|_\paramSpace^2 \le h(\wh\delta_0/\eps)^2\le c\wh\delta_0 \quad \text{with a small factor $c$,}
    $$ 
    so that
    $$
    \eps \|\dot \param_0\|_\paramSpace \le C \delta_0 + \tfrac14 \wh\delta_0.
    $$
    We further note that the obtained bound on $\lambda_1$ implies
    $$
    A(\param_0)\dot \param_0 -f(u_0) = - P_\eps(\param_0) G(u_0)^\top \lambda_1  + d_0 = O (\delta_0)+O(h\|\dot \param_0\|_\paramSpace^2),
    $$
    so that also
    $$
    \|A(\param_0)\dot \param_0 -f(u_0)\|_\calH \le C \delta_0 + \tfrac14 \wh\delta_0.
    $$
    Together with the estimate for $\eps \dot \param_0$ this shows that
    $$
    {\wh\delta_0}^{\,2} = \|A(\param_0)\dot \param_0 -f(u_0)\|_\calH^2 + \eps^2 \|\dot \param_0\|_\paramSpace^2 \le C'\delta_0^2 + \tfrac12 \wh\delta_0^{\,2},
    $$
    so that 
    $$
    \wh\delta_0 = O(\delta_0).
    $$
    Tracing the constants in the $O$-notation yields the stated result.
    \qed
\end{proof}

From this local error bound we again obtain a global $O(h+\delta)$ error bound as in 
Theorem~\ref{prop:glob-err-euler}, using Lady Windermere's fan with error propagation by the exact flow.

\section{Case study: Gradient systems} \label{sec:grad-sys}
In this section we study an approach to minimize a function via regularized parametric steepest descent. The resulting algorithm solves a regularized linear least squares problem for the increment in the parameters in every step and inherits the decay properties of classical (non-parametric and usually non-feasible) steepest descent up to the defect size. In contrast to standard gradient descent applied to the composition of the function with the parametrization map, which may have many local minima, the algorithm considered here is guaranteed to converge to the global minimum of a strongly convex function up to an error controlled by the defect size.

\subsection{Regularized parametric continuous-time gradient flow}
Consider the energy or loss function $V\colon\mathcal H \rightarrow \mathbb R$, with the associated unconstrained optimization problem 
\begin{equation}
 \min_{y\in\mathcal H} V(y) .    
\end{equation}
We seek approximations to the minimizer $y_*$, or generally functions that achieve a similar energy as the optimal energy $V_* = V(y_*)$. Formulating an evolution problem in the direction of the steepest energy descent yields the gradient flow, which reads
\begin{equation}\label{ivp-gradientflow}
\dot y = -\nabla V(y), \qquad y(0)=y_0.
\end{equation}
Here, the gradient $\nabla V \colon \mathcal H \rightarrow \mathcal H$ is understood to be the steepest descent direction of $V$ in the sense of the $\|\cdot\|_{\mathcal H}$-metric installed on $\mathcal H$. 

A sufficient condition on $V$ to ensure exponential decay of $V(y(t))$ to the minimal value $V_*$ as $t\to\infty$ for solutions $y(t)$ of the gradient flow is the Polyak--\L ojasiewicz inequality, that is, the existence of a constant $\mu>0$ such that  
\begin{align}\label{mu-PL}
    \tfrac{1}{2}\|\nabla V(y)\|_\calH^2 \ge \mu(V(y)-V_*)
    \quad \text{for all } y\in\calH.
\end{align}
This assumption ensures that the gradient flow and steepest descent methods reduce the loss function $V$ to its minimal value with the exponential rates, as is shown in the classical reference \cite{P63}. 

Approximating the evolution problem \eqref{ivp-gradientflow} by a parametric approach with $u(t)=\Phi(\param(t))\approx y(t)$ determines $\dot \param(t)$ by the least squares problem
\begin{equation}
	\label{reg-lsq-grad-flow}
	\delta(t)^2 := \| \dot u(t) + \nabla V(u(t)) \|_\calH^2 +\eps(t)^2 \| \dot \param(t) \|_\paramSpace^2 \quad\text{is minimal.}
\end{equation}
The parametric approximation \eqref{reg-lsq-grad-flow} of the gradient flow reduces the loss to its minimal value $V_*$ up to terms of the magnitude $\delta^2$ that are encountered along the approximation, as is shown in the following proposition. Similar results that were developed independently can be found in \cite{BCM25}.
\begin{proposition}
The parametric approximation of the gradient flow \eqref{reg-lsq-grad-flow} reduces the loss in the sense that
\begin{align*}
\dfrac{d}{dt} V(u(t))
 \le
-\tfrac{1}{2}\left\| \nabla V (u)\right\|_\calH^2 
+
\tfrac{1}{2}\delta(t)^2.
\end{align*}
Under the assumption \eqref{mu-PL}, we further have the explicit decay rates
\begin{align*}
 V(u(t)) - V_* \le e^{-\mu t}(V(u_0)-V_*)+\tfrac{1}{2}\int_{0}^te^{-\mu (t-s)}\delta(s)^2 \,\mathrm d s.
\end{align*}
\end{proposition}

\begin{proof}
The first inequality is given by using the chain rule of differentiation and the Cauchy-Schwarz inequality, which yields
\begin{align*}
\dfrac{d}{dt} V(u(t))
&=\left\langle \nabla V (u) ,\dot{u}\right\rangle_\calH
=
-\left\| \nabla V (u)\right\|_\calH^2 
+
\left\langle\nabla V (u), \dot u + \nabla V(u) \right\rangle_\calH
% \\ & \le
% -\left\| \nabla L (u)\right\|^2 
% +
% \left\| \nabla L (u)\right\|
% \left\|  \dot u + \nabla L(u) \right\|
\\ & \le
-\tfrac{1}{2}\left\| \nabla V (u)\right\|_\calH^2 
+
\tfrac{1}{2}\left\|  \dot u + \nabla V(u) \right\|_\calH^2
 \le
-\tfrac{1}{2}\left\| \nabla V (u)\right\|_\calH^2 
+
\tfrac{1}{2}\delta(t)^2.
\end{align*}
Inserting the Polyak--\L ojasiewicz inequality \eqref{mu-PL} then yields
\begin{align*}
\dfrac{d}{dt} V(u(t))
& \le
- \mu (V(u(t))-V_*)+\tfrac{1}{2}\delta(t)^2.
\end{align*}
The statement is now the direct consequence of the same Gronwall estimate that has been used in Proposition~\ref{prop:delta}.
\qed
\end{proof}

\subsection{Regularized parametric steepest descent}
Applying the regularized parametric explicit Euler method of Section~\ref{sec:time-disc} to the gradient flow gives a parametric steepest descent method, which computes approximations $u_n=\Phi(\theta_n)$. The method defines $\dot\param_n$ via
\begin{equation}\label{euler-gradient-flow}
\delta_n^2 := \| \Phi'(\param_n)\dot \param_n + \nabla V(u_n)\|_\calH^2 + \eps_n^2 \| \dot \param_n \|_\paramSpace^2 \quad\text{is minimal,}
\end{equation}
and updates the parameters by setting
\begin{equation}
\param_{n+1} = \param_n+h\dot\param_n.
\end{equation}
To ensure the convergence of the classical (non-parametric)  steepest descent method to a minimizer, an additional property on $V$ to the Polyak--\L ojasiewicz inequality \eqref{mu-PL} is required. We assume that the second derivative of $V$ is bounded by a constant $M$, such that for all $x,y\in \calH$, we have the inequality
\begin{align}\label{eq:second-derivative-M}
V(y) \le V(x) + \left \langle \nabla V(x), y-x \right \rangle + \dfrac{M}{2} \|y-x\|_\calH^2 .
\end{align}

\begin{theorem}\label{thm:eulerGD}
 Under the step size restriction $hM \le 1/4$, we obtain for a single parametric steepest descent step
 the decay estimate
\begin{align*}
V(u_{1})- V(u_0)&\le  -\frac12 h \left\| \nabla V(u_0) \right\|^2_{\calH}+\dfrac{5h}{4}
\left(\delta_0 + \dfrac{\beta}{2}\dfrac{h\delta_0^2}{\eps_0^2} \right)^2
,
\end{align*}
 where $\beta$ is a bound of the second derivative of $\Phi$ in a neighbourhood of $\param_0$.
 Under the additional step size restriction \eqref{euler-h}, viz.,
 \begin{align}\label{step-size-restriction}
  h\delta_n \le c\,\eps_n^2,
\end{align}
we therefore obtain that the steepest descent step reduces the energy functional as long as 
%$$ \|\nabla V(u_0)\|_{\calH} \ge \frac{\sqrt{10}}{2}\left(1+\frac{\beta c}{2}\right) \delta_0. $$
$$ \|\nabla V(u_0)\|_{\calH} \ge \sqrt{\frac{c_1}{2}}\delta_0$$
with the constant 
$c_1 = \tfrac54\left(1+c\beta/2 \right)^2.$
Under the Polyak--\L ojasiewicz inequality \eqref{mu-PL} and the step size restrictions $h\mu\le 1$, $hM \le 1/4$ and \eqref{step-size-restriction}, we further have the exponential decay estimate 
\begin{align*}
V(u_{n})-V_* 
\le
(1-h\mu)^n (V(u_0)-V_*) +c_1 h
\sum_{j=1}^{n}(1-h\mu)^{j-1}\delta_{n-j}^2. 
\end{align*}
%$c_1 = \left(2+c\beta \right)^2.$
% \begin{align*}
%     V(u_n)-V_* \le \left(1-h\mu\right)^n (V(u_0)-V_*) +C\sum_{j=0}^n(1-h\mu)^{n-j} \delta_j.
% \end{align*}   
\end{theorem}
\begin{proof}

\noindent \emph{(i) Decay estimate.} Inserting $u_n= \Phi(\theta_n)$ and $u_{n+1}=\Phi(\theta_{n+1})$ into \eqref{eq:second-derivative-M} yields
\begin{align}\label{eq:second-derivative-M-un-inserted}
V(u_{n+1})-V(u_n) \le \left \langle \nabla V(u_n), u_{n+1}-u_n \right \rangle + \dfrac{M}{2} \|u_{n+1}-u_n\|_{\mathcal{H}}^2 .
\end{align}
The difference between $u_{n+1}$ and $u_n$ is characterized by the local error estimate for the explicit Euler method of Lemma~\ref{lem:loc-err-euler}, which gives
% \begin{align*}
% u_{n+1}-u_n
% =
%  -h\nabla V(u_n) + O(h\delta_n) + O\biggl(\Bigl(\frac{h\delta_n}{\eps_n}\Bigr)^2\biggr).
% \end{align*}
%More precisely, we have
\begin{align*}
\| u_{n+1}-u_n+h\nabla V(u_n)\|_{\calH}
\le h\delta_n + \dfrac{\beta}{2}\dfrac{h^2\delta_n^2}{\eps_n^2}.
\end{align*}
Inserting this estimate for $n=0$ into the first summand of \eqref{eq:second-derivative-M-un-inserted} and applying Young's inequality ($ab\le \frac14 a^2+b^2$ for $a,b\in\R$) yields
\begin{align*}
\left\langle \nabla V(u_0), u_{1}-u_0 \right\rangle 
&\le-h \left\| \nabla V(u_0) \right\|^2_{\calH}+
\left\| \nabla V(u_0) \right\|_{\calH}
\left(h\delta_0 + \dfrac{\beta}{2}\dfrac{h^2\delta_0^2}{\eps_0^2} \right)
\\& \le
-\dfrac{3}{4} h \left\| \nabla V(u_0) \right\|^2_{\calH}+
 h
\left(\delta_0 + \dfrac{\beta}{2}\dfrac{h\delta_0^2}{\eps_0^2} \right)^2.
\end{align*}
For the second summand of the right-hand side in \eqref{eq:second-derivative-M-un-inserted}, we use the bound
\begin{align*}
&\dfrac{M}{2}\|u_{1}-u_0\|_{\mathcal{H}}^2 
% \\&\le 
%  \|h\nabla V(u_n)\|_{\calH}^2 
% + \|h\nabla V(u_n)\|_{\calH}\left(h\delta_n + \dfrac{\beta}{2}\dfrac{h^2\delta_n^2}{\eps_n^2} \right) +\left(h\delta_n + \dfrac{\beta}{2}\dfrac{h^2\delta_n^2}{\eps_n^2} \right)^2
\le 
M \|h\nabla V(u_0)\|_{\calH}^2 
+
Mh^2\left(\delta_0 + \dfrac{\beta}{2}\dfrac{h\delta_0^2}{\eps_0^2} \right)^2 . 
\end{align*}
Inserting both bounds on the right-hand side of \eqref{eq:second-derivative-M-un-inserted} then gives, with the step size restriction $Mh\le \tfrac{1}{4}$, the estimate
\begin{align*}
V(u_{1})-V(u_0) 
%\le -h \left\| \nabla V(u_n) \right\|^2_{\calH}+
%\left\| \nabla V(u_n) \right\|_{\calH}
%\left(h\delta_n + \dfrac{\beta}{2}\dfrac{h^2\delta_n^2}{\eps_n^2} \right)
% +M \|h\nabla V(u_n)\|_{\calH}^2 
%+M\left(h\delta_n + \dfrac{\beta}{2}\dfrac{h^2\delta_n^2}{\eps_n^2} \right)^2
&\le \left(-\frac34 +M h\right)h \left\| \nabla V(u_n) \right\|^2_{\calH}+(h+Mh^2)
\left(\delta_0 + \dfrac{\beta}{2}\dfrac{h\delta_0^2}{\eps_0^2} \right)^2
% \\&\le -\frac12 h \left\| \nabla V(u_n) \right\|^2_{\calH}+h
% \left(\delta_n + \dfrac{\beta}{2}\dfrac{h\delta_n^2}{\eps_n^2} \right)^2
% \\& 
% +Mh^2\left(\delta_n + \dfrac{\beta}{2}\dfrac{h\delta_n^2}{\eps_n^2} \right)^2
\\&\le -\frac12 h  \left\| \nabla V(u_0) \right\|^2_{\calH}+\dfrac{5h}{4}
\left(\delta_0 + \dfrac{\beta}{2}\dfrac{h\delta_0^2}{\eps_0^2} \right)^2.
\end{align*}

\emph{(ii) Exponential decay estimate for strongly convex $V$.} Using the Polyak--\L ojasiewicz inequality on the right-hand side of the decay estimate yields, for arbitrary $n\ge 0$,
\begin{align*}
V(u_{n})-V(u_{n-1}) 
\le
-h\mu (V(u_{n-1})-V_*) +\dfrac{5h}{4}
\left(\delta_{n-1} + \dfrac{\beta}{2}\dfrac{h\delta_{n-1}^2}{\eps_{n-1}^2} \right)^2.
\end{align*}
Subtracting $V_*$ from both sides, rearranging and inserting the step size restriction at the index $n-1$ then gives
\begin{align*}
V(u_{n})-V_* 
\le
(1- h\mu ) (V(u_{n-1})-V_*) +\dfrac{5h}{4}\left(1+\dfrac{c\beta}{2} \right)^2
\delta_{n-1}^2.  
\end{align*}
Repeating this inequality on the right-hand side then gives the result.
\qed
\end{proof}

% \ \\ \ \\ \ \\
% % \begin{align*}
% % V(u_{n+1})-V(u_n) \le  -h \left\| \nabla V(u_n)\right \|_\calH^2 + M\|u_{n+1}-u_n\|_{\mathcal{H}}^2 .
% % \end{align*}
% \begin{align*}
% V(u_{n+1})-V(u_n) &\le  (-h+\tfrac{1}{2}Mh^2) \left\| \nabla V(u_n)\right \|_\calH^2 +O(h\delta_n) + O\biggl(\Bigl(\frac{h\delta_n}{\eps_n}\Bigr)^2\biggr)
% \\&\le
% (-2h+Mh^2) \mu ( V(u_n)-V_*)+O(h\delta_n) + O\biggl(\Bigl(\frac{h\delta_n}{\eps_n}\Bigr)^2\biggr).
% \end{align*}

% We obtain
% \begin{align*}
% V(u_{n+1})\le \left(1-h\mu\right)V(u_n)+C \delta_n,
% \end{align*}

% If the weaker step size restriction holds for $h=\dfrac{1}{M}$, we obtain
% \begin{align*}
% V(u_{n+1})\le \left(1-\dfrac{\mu}{M}\right)V(u_n)+C h\delta_n
% \end{align*}

%\subsection{Fitting a function}
\medskip\noindent
{\it Example: Fitting a function.}
A direct application is the construction of parameters $\theta_*$ such that $\Phi(\theta_*)\approx g$, for some $g\in \mathcal H$ of interest. The natural loss function is then given by 
\begin{align*}
V_g(y) &= \tfrac12\|y-g\|_{\mathcal{H}}^2.  
\end{align*}
The minimum is attained for $y_*=g$ with $V_g(y_*) = V_g(g) = 0$.
The gradient equals $\nabla V_g(y) = y-g$ and the gradient flow thus becomes
\begin{align*}
 \dot y = - (y-g).
\end{align*}
The loss function therefore fulfills the Polyak--\L ojasiewicz inequality \eqref{mu-PL} with $\mu=1$ and its second derivative is bounded by $M=1$.

\bigskip
We close this section with a comparison of standard 
steepest descent for the parametrized loss function $V\!\circ\Phi$ and the regularized parametric steepest descent for the loss function $V$, as considered above. Let again $u_n=\Phi(\theta_n)$.

-- Standard steepest descent for $V\!\circ\Phi$ updates the parameters via
\begin{align*}
    \theta_{n+1} = \theta_n - h_n \,\Phi^\prime(\theta_n)^*\nabla V(u_n)
\end{align*}
and can easily get stuck in one of possibly many local minima of the parametrized loss function $V\!\circ \Phi$, even if $V$ is strongly convex.

-- In the regularized parametric steepest descent for $V$, we update the parameters via
\begin{align*}
    \theta_{n+1} = \theta_n - h_n \left(\Phi^\prime(\theta_n)^*\Phi^\prime(\theta_n)+ \eps_n^2 I\right)^{-1}\Phi^\prime(\theta_n)^*\nabla V(u_n),
\end{align*}
if the regularized least squares problem is solved by the normal equations.
Theorem~\ref{thm:eulerGD} shows that the regularized parametric steepest descent finds the global minimizer of a strongly convex function $V$ up to an error that depends entirely on the regularization and on approximation properties of the parametrization indicated by the defects in the regularized linear least squares problems.

\section{Case study: Schr\"odinger equation}
\label{sec:psi}

The time-dependent Schr\"odinger equation is arguably the evolution equation for which nonlinear approximations have been first and most often used, ever since Dirac's paper of 1930 \cite{Dir30}. Gaussians and tensor networks are nowadays the most prominent examples of nonlinear approximations in quantum dynamics. As a partial differential equation with an unbounded operator, the Schr\"odinger equation does not fall into the Lipschitz framework that was mostly considered so far. In this section we study what remains and what needs to be changed in the regularized approach.

\subsection{Preparation}
The Schr\"odinger equation determines the complex-valued wave function $\psi(x,t)$ that depends on spatial variables $x\in \R^d$ and time $t$:
\begin{equation}\label{tdse}
\iu  \dot \psi = -\Delta \psi + V\psi,
\end{equation}
where $\iu$ is the imaginary unit, $\dot \psi=\partial_t \psi$ is the time derivative,  $\Delta$ is the Laplacian on $\R^d$ and $V=V(x)$ is a real-valued potential that multiplies the wave function. % and is assumed to be bounded in the following. 
The Schr\"odinger equation is considered as an evolution equation on the Hilbert space $\calH=L^2(\R^d)$ for the wave function $\psi(t)=\psi(\cdot,t)\in \calH$.
\bch 
In this section we write $\|\cdot\|= \|\cdot\|_\calH = \|\cdot\|_{L^2}$.
\ech

Consider a continuously differentiable map $\Phi$ from a parameter space $\paramSpace$ into $\calH$. We aim to approximate
$$
\psi(t)\approx u(t)=\Phi(\param(t)) \in \calH \quad\text{ for some }\ \param(t)\in\paramSpace
$$
by the regularized dynamical nonlinear approximation \eqref{reg-lsq} with the linear operator $f(u)= \iu \Delta u - \iu Vu$.
Since the Laplacian is an unbounded operator, the Lipschitz framework of the previous sections does not apply here. However, we still
have the one-sided Lipschitz condition with $\ell=0$ and from this we obtain the {\it a posteriori} error bound of Proposition~\ref{prop:delta}, with the same proof. 

To obtain an {\it a priori} error bound, we need to modify the construction of the regularized approximation $u=\Phi(\param)$. % and of the reference approximation $u_*=\Phi(\param_*)$. 
We will  use the property that the Laplacian maps into the tangent space:
\begin{equation}\label{Lap}
\text{If $u=\Phi(\param)$, then $\Delta u= \Phi'(\param)\param^\Delta$ for some $\param^\Delta\in \paramSpace$.}
\end{equation}
This holds true for Gaussians and tensor networks but not for neural networks. We assume \eqref{Lap} throughout this section.

\bch
\subsection{Modified regularized dynamical approximations}
\subsubsection{First modification} 
\ech
Instead of \eqref{reg-lsq}, we now choose (omitting the argument $t$) 
$\dot u =\Phi'(\param)\dot \param$ and $\dot \param\in\paramSpace$ such that
\begin{equation}\label{vp}
\dot u - \iu \Delta u = v \quad\text{and}\quad \dot \param - \iu \param^\Delta = p,
\end{equation}
where $v=\Phi'(\param)p$ with $p\in \paramSpace$ is chosen such that
\begin{equation}
\label{reg-lsq-psi}
\delta^2 := \| v + \iu Vu \|^2 +\eps^2 \| p \|_\paramSpace^2 \quad\text{is minimal.}
\end{equation}
Inserting \eqref{vp} into \eqref{reg-lsq-psi} and comparing with \eqref{reg-lsq}, we find that the term $\eps^2\|\dot \param\|^2$ in \eqref{reg-lsq} is now replaced by $\eps^2\|  \dot \param - \iu \param^\Delta \|^2$, everything else being equal.
In contrast to \eqref{reg-lsq},  the free Schr\"odinger equation (i.e., with $V=0$) is solved exactly with \eqref{vp}--\eqref{reg-lsq-psi} for every $\eps>0$. Indeed, $p=0$ provides $\dot \param = i\param^\Delta$ and $\dot u = i\Delta u$. 
%For $\eps=0$,  \eqref{reg-lsq} and  \eqref{vp}--\eqref{reg-lsq-psi} are equivalent under condition \eqref{Lap}. 
We have 
$$
\partial_t (u-\psi) = \iu\Delta (u-\psi) -\iu V(u-\psi) + d  \quad\text{ with }\quad \| d \| \le \delta,
$$
and as in the proof of Proposition~\ref{prop:delta} (with $\ell=0$), we obtain the {\it a posteriori} error bound
\begin{equation}\label{err-u-psi}
\| u(t)-\psi(t) \| \le \int_0^t  \delta(s)\, ds.
\end{equation}
\bch
\subsubsection{Second modification}
With the quasi-projection  
$$P_\eps(\param) = A(\param) M_\eps(\param)^{-1} A(\param)^*,
$$ 
where $A=\Phi'$ and $M_\eps=A^* A + \eps^2 Q$ with the hermitian positive definite matrix $Q$ that defines $\|\param\|_\paramSpace^2=\param^*Q\param$,
we have that the regularized parametric approximation $u(t)$ of \eqref{reg-lsq} with the linear operator $f(u)= \iu \Delta u - \iu Vu$ satisfies
\begin{equation}
    \label{u-reg-lsq}
    \iu \dot u = P_\eps(\param) \bigl(-\Delta u + Vu\bigr),
\end{equation}
whereas $u(t)$ of \eqref{vp} satisfies
\begin{equation}
    \label{u-reg-lsq-vp}
    \iu \dot u = -\Delta u + P_\eps(\param)Vu.
\end{equation}
In neither of these two equations, the right-hand side is a self-adjoint operator acting on $u$, as would be characteristic of Schr\"odinger equations. We therefore modify \eqref{u-reg-lsq} to
\begin{equation}
    \label{u-reg-lsq-sa}
    \iu \dot u = -\Delta u + P_\eps(\param)V P_\eps(\param) u,
\end{equation}
which now has the self-adjoint parameter-dependent Hamiltonian operator $H_\eps(\param)=-\Delta + P_\eps(\param)V P_\eps(\param)$ on the right-hand side.

\ech

\begin{remark}
Multiplying a sum of complex Gaussians $u=\Phi(\param)= \sum_j \varphi(z_j)$ 
by a subquadratic potential $V$, we have
\[
Vu = \sum_j  V\varphi(z_j) = \sum_j  (U_j+W_j)\varphi(z_j),  
\]
where $U_j$ denotes the second order Taylor polynomial of $V$ centered around the position center of the 
$j$th Gaussian and $W_j$ the cubic remainder. Therefore, 
\[
\text{$Vu = \Phi'(\param)\param^U + \chi(\param)$ for some $\param^U\in\paramSpace$}
\]
and $\chi(\param)=\sum_j W_j\varphi(z_j)$ with $\|\chi(\param)\|\le \beta_3 \| (1+|x|^2) u\|$ 
for some constant $\beta_3>0$ depending on third order derivative bounds of $V$. Working with $\param^\Delta+\param^U$ 
instead of $\param^\Delta$ extends the above approximation and makes it exact for harmonic oscillators. 
\end{remark}

\subsection{A priori error bound}
We can bound the defect size $\delta(t)$ by a quantity that measures the uniform approximability of $\dot\psi-\iu\Delta\psi$ in the tangent spaces $T_{(u,\param)}\calM$ for all $u=\Phi(\param)$ in a neighbourhood of $\psi(t)$. We fix a radius $\rho>0$ and define
\begin{equation}
    \label{delta-rho-psi}
    \bar\delta_\rho(t)^2 := \sup_{\param\in\paramSpace: \| \Phi(\param)-\psi(t) \| \le \rho}\ 
    \min_{\dot \param \in \paramSpace} \Bigl( \|\Phi'(\param)\dot \param - (\dot\psi-\iu\Delta\psi) \|^2 + \eps^2 \| \dot \param \|_\paramSpace^2 \Bigr).
\end{equation}
We can bound the defect size $\delta(t)$ of \eqref{reg-lsq-psi} in terms of $\bar\delta_\rho(t)$. 
\begin{lemma}
\label{lem:delta-bound-psi}
%Assume that the potential $V$ is bounded. 
Provided that $\| u(t)-\psi(t) \| \le \rho$, we have
$$
 \  \delta(t) \le \bar\delta_\rho(t) + \| Vu(t)-V\psi(t)) \|.
$$
\end{lemma}

\begin{proof} The proof is similar to that of Lemma~\ref{lem:delta-bound}.
We omit the argument $t$ in the following. We have $u=\Phi(\param)$ and $\dot u=\Phi'(\param)\dot \param$.
Let $p_+\in \paramSpace$ be such that 
$$
\|\Phi'(\param)p_+ - (\dot\psi-\iu\Delta\psi) \|^2 + \eps^2 \| p_+ \|_\paramSpace^2 \quad\text{is minimal.}
$$
By \eqref{tdse}--\eqref{reg-lsq-psi},
\begin{align*}
\delta^2 &= \| \Phi'(\param)p + \iu Vu\|^2 + \eps^2 \| p \|_\paramSpace^2 
\\[1mm]
&\le \| \Phi'(\param)p_+ + \iu Vu\|^2 + \eps^2 \| p_+ \|_\paramSpace^2 
\\
&\le \Bigl( \|  \Phi'(\param)p_+  - (\dot\psi-\iu\Delta\psi) \| + \|- \iu V\psi +\iu Vu  \| \Bigr)^2 + \eps^2 \| p_+ \|_\paramSpace^2
\\
%&\le \delta_*^2 + 2 \delta_* \, \| V(\psi-u) \| + \| V(\psi-u) \| ^2
%\\
&\le \Bigl( \bar\delta_\rho + \| Vu -V\psi  \| \Bigr)^2,
\end{align*}
which yields the result.
\qed
\end{proof}

We construct a reference approximation $u_*=\Phi(\param_*)$ from the exact wave function $\psi$ by choosing
$\dot u_* =\Phi'(\param_*)\dot \param_*$ and $\dot \param_*\in\paramSpace$ such that
\begin{equation}\label{vp-star}
\dot u_* - \iu \Delta u_* = v_* \quad\text{and}\quad \dot \param_* - \iu \param_*^\Delta = p_*,
\end{equation}
where $v_*=\Phi'(\param_*)p_*$ with $p_*\in \paramSpace$ is chosen as a regularized best approximation to $\dot\psi-\iu\Delta\psi$ in the tangent space:
\begin{equation}
\label{reg-lsq-star-psi}
\delta_*^2 := \| v_* - (\dot\psi-\iu\Delta\psi) \|^2 +\eps^2 \| p_* \|_\paramSpace^2 \quad\text{is minimal.}
\end{equation}
With \eqref{reg-lsq-star-psi} we have 
$$
\partial_t (u_*-\psi) = \iu\Delta (u_*-\psi) + d_*  \quad\text{ with }\quad \| d_* \| \le \delta_*,
$$
and as before %in the proof of Proposition~\ref{prop:delta} 
it follows that
\begin{equation}\label{err-u-star-psi}
\| u_*(t)-\psi(t) \|  \le \int_0^t  \delta_*(s)\, ds \le \int_0^t  \bar\delta_\rho(s)\, ds
\end{equation}
as long as this is bounded by $\rho$.
The error of the numerical approximation $u(t)=\Phi(\param(t))$ is bounded by a multiple of the bound in \eqref{err-u-star-psi}.
\begin{proposition}
\label{prop:delta-star-psi}
If condition \eqref{Lap} is satisfied and $\ \sup_x |V(x)| \le \nu$, then the error of $u(t)$ defined by \eqref{vp}--\eqref{reg-lsq-psi} is bounded by
$$
\| u(t)-\psi(t) \| \le e^{\nu t}\int_0^t   \bar\delta_\rho(s)\, ds
$$
as long as this is bounded by $\rho$.
\end{proposition}

\begin{proof}
The bound follows from Lemma~\ref{lem:delta-bound-psi} inserted into \eqref{err-u-psi} and using the Gronwall lemma.
\qed
\end{proof}

\subsection{Energy and norm conservation}
With the Hamiltonian $H=-\Delta + V$, which is a self-adjoint linear operator on $\calH=L^2(\R^d)$ with domain $D(H)=H^2(\R^d)$, the total energy is defined as $\langle u,Hu \rangle_\calH$ for $u\in D(H)$. The energy is conserved along the exact wave function $\psi(t)\in D(H)$ of \eqref{tdse}, since (omitting in the following the argument $t$ and the subscript $\calH$ in the inner product)
$$
\frac{d}{dt}\langle \psi,H\psi \rangle = 2\,\Re\langle  H\psi, \dot \psi \rangle = 2\,\Re\langle  H\psi, -\iu H\psi \rangle =0.
$$
We show that the energy is also conserved along the regularized approximation \eqref{reg-lsq}. 
% We adapt the notation of Section~\eqref{sec:con} to the complex setting and let 
%  $P_\eps(\param) = A(\param) M_\eps(\param)^{-1} A(\param)^*$ with $A=\Phi'$ and $M_\eps=A^* A + \eps^2 I$. 
 % We note that  \eqref{reg-lsq} yields
 % $\iu \dot u =   P_\eps(\param) H u$. 
Since both $H$ and $P_\eps(\param)$ are self-adjoint,
we have for the method \eqref{reg-lsq} or equivalently \eqref{u-reg-lsq}
(omitting the arguments %$u(t)=Phi(\param(t))$ or 
$\param(t)$ and $t$)
$$
\frac{d}{dt}\langle u,Hu \rangle = 2\,\Re\langle  Hu, \dot u \rangle = 2\,\Re\langle  Hu, -\iu P_\eps Hu \rangle =0.
$$
However, the modified approximation $u$ defined by \eqref{vp}--\eqref{reg-lsq-psi} or equivalently \eqref{u-reg-lsq-vp}
has
%satisfies the differential equation $\iu\dot u = -\Delta u + P_\eps Vu$ and thus  
\begin{align*}
\frac{d}{dt}\langle u,Hu \rangle &= 2\,\Re\langle  Hu, \dot u \rangle 
= 2\,\Re\langle  - \Delta u + Vu, \iu \Delta u - \iu P_\eps Vu \rangle 
\\[-1mm]
&= 2\,\Re\langle   \Delta u, \iu P_\eps Vu \rangle + 2\,\Re\langle  Vu , \iu \Delta u \rangle 
\\
&= 2\,\Re\langle   \Delta u,  \iu (I-P_\eps)V u  \rangle,
\end{align*}
% \begin{align*}
% \frac{d}{dt}\langle u,Hu \rangle &= 2\,\Re\langle  Hu, \dot u \rangle 
% = 2\,\Re\langle  Hu, \iu \Delta u - \iu P_\eps Vu \rangle 
% \\[-1mm]
% &= 2\,\Re\langle  Hu, \iu (I-P_\eps)\Delta u - \iu P_\eps Hu \rangle
% = 2\,\Re\langle  Hu, \iu (I-P_\eps)\Delta u  \rangle 
% \\
% &= 2\,\Re\langle  Vu, \iu (I-P_\eps)\Delta u  \rangle,
% \end{align*}
which in general is nonzero.

\bch 
The $L^2$-norm $\|u(t)\|$ is preserved only by the second modified approximation \eqref{u-reg-lsq-sa}: since 
$P_\eps V P_\eps$ is self-adjoint, we have
$$
\frac{d}{dt}\| u \|^2 = 2 \,\Re \langle u, \dot u \rangle = 2\, \Im \langle u, -\Delta u+ P_\eps V P_\eps u \rangle =0.
$$
However, this method also does not preserve energy.
Norm and energy conservation can be enforced simultaneously as described in Section~\ref{sec:con}.

\subsection{Time discretization by regularized parametric Strang splitting}
\subsubsection{Abstract Strang splitting}
Splitting methods that separate the kinetic and potential terms in the Schr\"odinger equation \eqref{tdse} are standard methods for approximately propagating the wave function $\psi(t)$. The Strang splitting approximates $\psi(t_n)\approx \psi_n$ with step size $h$ by
\begin{equation}\label{strang}
\psi_{n+1}= \exp\bigl(\tfrac12 h \, \iu \Delta \bigr)\, \exp \bigl( -h \, \iu V \bigr) \, \exp\bigl(\tfrac12 h \, \iu \Delta \bigr)
\psi_n.
\end{equation}
It was shown in \cite{JL00} that in the case of a bounded twice continuously differentiable potential $V$ and a sufficiently regular wave function $\psi$,
\begin{equation}
    \label{strang-err}
    \| \psi_n - \psi(t_n) \|_{L^2} \le c \,t_n\, h^2 \max_{0\le t \le t_n} \| \psi(t) \|_{H^2},
\end{equation}
where $H^2$ is the second-order $L^2$-based Sobolev space.

\subsubsection{Regularized parametric Strang splitting}

\hskip 6mm
(i) Starting from $u_0=\Phi(\param_0)$, we first compute 
$$
u_0^+ = \Phi(\param_0^+) = \exp\bigl(\tfrac12 h \, \iu \Delta \bigr)u_0,
$$ 
which is feasible under the assumption~\eqref{Lap}. For Gaussians, simple explicit formulas related to classical mechanics are known to compute the parameters $\param_0^+$; see e.g. \cite[Prop.\,3.18]{LL20}.

(ii) We approximate $u_1^-= \Phi(\param_1^-) \approx \exp \bigl( -h \, \iu V \bigr) u_0^+$, where
$\param_1^-$ is the result of one step of a Runge--Kutta method (of order $p\ge 2$) applied to the differential equation for $\param(t)$ that results from determining $\dot \param(t)$ by the regularized linear least squares problem (omitting the argument $t$)
\begin{equation}
\label{reg-lsq-V}
\delta^2 := \| \iu \Phi'(\param) \dot \param - V \Phi(\param) \| ^2 + \eps^2 \| \dot \param \|_\paramSpace ^2 = \min!
\end{equation}
We then set $u_1^-=\Phi(\param_1^-)$.

(iii) As in step (i), we compute
$$
u_1 = \Phi(\param_1) = \exp\bigl(\tfrac12 h \, \iu \Delta \bigr)u_1^-.
$$ 
This scheme is repeated to compute consecutively $u_2$, $u_3$, etc.

\subsubsection{Error bound}
We prove an error bound over bounded time intervals by the following steps:
\begin{enumerate}
    \item We bound the difference of the results $\psi_1$ and $u_1$ obtained after one step of the abstract (non-parametrized) Strang splitting and the regularized parametrized Strang splitting, respectively, starting from the same initial value $\psi_0=u_0=\Phi(\param_0)$.
    \item Using Lady Windermere's fan with error propagation by the stable abstract (non-parametrized) Strang splitting,
    we bound $u_n-\psi_n$.
    \item The bound for the error $u_n - \psi(t_n)= (u_n-\psi_n) + (\psi_n - \psi(t_n))$ then follows from the known error bound \eqref{strang-err} of the abstract Strang splitting.
\end{enumerate}
Concerning item 1., we have the following bound.

\begin{lemma}
    \label{lem:psi-err-local}
    Assume that the potential $V$ is bounded by $\nu$. Under the stepsize restriction \eqref{rk-h}, we then have 
    $$
    \| u_1 - \psi_1 \|_{L^2} \le c_1 \,h\delta_0 + c_2 \,h^3,
    $$
    where $\delta_0=\max_{i=1,\ldots,s} \delta_{0,i}$ with the defect sizes $\delta_{0,i}$ of the regularized Runge--Kutta method used in substep (ii) above; see \eqref{rk-b} with $f(u)=-\iu Vu$. Moreover, $c_1$ and $c_2$ are independent of $\delta_0$ and $h$, and $c_1$ depends linearly on a bound of the second derivative of $\Phi$, and 
    $c_2$ is proportional to $\nu^3$.
\end{lemma}

\begin{proof}
We have
$$
u_1 - \psi_1 = \exp\bigl(\tfrac12 h \, \iu \Delta \bigr) 
\bigl(u_1^- - \exp \bigl( -h \, \iu V \bigr) u_0^+ \bigr)
$$
and since Lemma~\ref{lem:loc-err-rk} shows that for a Runge--Kutta method of order at least 2 we have the local error bound
$$
\| u_1^- - \exp \bigl( -h \, \iu V \bigr) u_0^+ \| = O(h\delta_0 + h^3),
$$
the result follows.
\qed
\end{proof}
Concerning item 2., let us write the abstract Strang splitting as $\psi_{n+1} = S_h \psi_n$, where $S_h$ is the composition of the three exponentials appearing in \eqref{strang}. Since each of the three exponentials is unitary, so is $S_h$ and hence also $S_h^n$ for every $n$. This implies that any two Strang splitting sequences $\psi_n=S_h^n \psi_0$ and $\widetilde \psi_n=S_h^n\widetilde \psi_0$ with starting values $\psi_0$ and $\widetilde \psi_0$ satisfy
$$
\| \psi_n - \wt \psi_n \| \le \| \psi_0 - \wt \psi_0 \|.
$$
Via Lady Windermere's fan with propagation of the local errors by $S_h$, we therefore obtain from Lemma~\ref{lem:psi-err-local} that
$$
\| u_n - \psi_n \| \le \sum_{j=1}^n\| u_j - S_h u_{j-1} \| \le nh (c_1\delta + c_2h^2) ,
$$
where $\delta$ is the maximal defect size appearing in the regularized least squares problems~\eqref{reg-lsq-V} at the Runge--Kutta stages.
Combining this bound with the error bound \eqref{strang-err} of the abstract Strang splitting, we arrive at the following global error bound.

\begin{theorem}
    Assume that the potential $V$ is bounded and twice continuously differentiable with bounded derivatives, and assume that the wave function $\psi(t)$ of \eqref{tdse} with initial value
    $\psi_0=u_0=\Phi(\param_0)$ has a bounded $H^2$-norm for $0\le t \le \bar t$. 
    Under the stepsize restriction \eqref{rk-h}, the error of the regularized parametrized Strang splitting, described in (i)--(iii) above, is bounded by
    $$
    \| u_n - \psi(t_n) \|_{L^2} \le t_n \bigl(C_1 \delta + C_2 h^2\bigr),
    $$
    where $C_1$ and $C_2$ are independent of $n$ and $h$ and the wave function $\psi$, and where $\delta$ is the maximal defect size appearing in the regularized least squares problems~\eqref{reg-lsq-V} at the Runge--Kutta stages.
    $C_1$ depends on a bound of the second derivative of $\Phi$, and $C_2$ depends on bounds of $V$ and its first two derivatives and on a bound of the $H^2$-norm of $\psi(t)$ on the considered time interval.
    \qed
\end{theorem}

\ech

\section{Numerical experiments}
\label{sec:num}

We close the paper with a collection of numerical experiments, both for the approximation of flow maps for ODEs and the Schrödinger equation.
\subsection{Approximating the flow map of a Lotka--Volterra model}
As a simple nonlinear initial value problem, we consider the classical predator--prey model
\begin{equation}\label{lv}
\begin{aligned}
\dot{x} &= \alpha x - \beta x y.
\\
\dot{y} &= \delta xy - \gamma y,
\end{aligned}
\end{equation}
where $\alpha,\beta,\gamma,\delta$ are positive constants, which are all set to 1 in our experiments.
The region of interest in our experiments is the square $D = \left[\tfrac{1}{2},\frac{5}{2}\right]^2$. 

The flow map $\varphi_t:D\to \R^2$ at time $t$, which to every initial value $(x_0,y_0)\in D$ associates the corresponding solution value $(x(t),y(t))$ of \eqref{lv}, is considered as an element of 
the Hilbert space $\calH = L^2(D)^2$. It satisfies the differential equation on $\calH$
$$
\frac d{dt}\varphi_t = f (\varphi_t), \qquad \varphi_0= \text{Id},
$$
where $f$ is given by the right-hand side of \eqref{lv} and $\text{Id}$ is the identity on~$\calH$.

We use the Tensorflow library to approximate the flow map $\varphi_t$ by a small feedforward neural network with three fully connected hidden layers, each with a depth of four neurons.  Overall, the network architecture requires only $62$ parameters. \textcolor{black}{For a smooth parametrization $\Phi(\param)$,} we use the sigmoid function
\begin{align*}
\sigma(x) = \frac{e^x}{1+e^x} 
\end{align*}
as the  activation function on each layer. (We note that the popular Relu function $\max(x,0)$ is not covered by our theory, which requires a twice continuously differentiable parametrization.)
The metric installed on the weight space $\mathcal Q$ is the standard euclidian norm.

As the initial nonlinear parametrization at $t=0$, we require a network that approximates the identity on $D$ with a high accuracy. For our experiments, this was realized by pretraining an initial approximation $u_0 \colon D\rightarrow \mathbb R^2$ with a standard optimizer and then applying the regularized  procedure \eqref{reg-lsq} to the initial value-problem
$$\dot{u}(\cdot,t) =\text{Id} - u_0$$
with time-independent right-hand side.
At $t = 1$, we then obtain an approximation $u(\cdot,1) \approx \text{Id}$. For the presented experiments, we used the classical Runge--Kutta method of order $4$ with the regularization parameter $\varepsilon=10^{-6}$ and $N=2000$ time steps.

\begin{remark}[Connection to the Gau\ss--Newton method] We note that the approach above can be used to construct neural networks (or any nonlinear parameterization) approximating any given arbitrary function $g$, instead of the identity $\text{Id}$. The implementation of such methods generalizes the Gau\ss--Newton method applied to $\|\Phi(\param) - g\|=\min$, which would be obtained by applying the Euler method (with $\varepsilon=0$ and $h=1$) to the differential equation that results from the (non-regularized) least squares problem
$%\begin{align*}
\| \Phi'(\param)\dot \param - (g- \Phi(\param))\|= \min
$.%\end{align*}
\end{remark} 

The assembly of the Jacobian $\Phi'(\param_{n,i})$, for the given parameters at the internal stages $\param_{n,i}$ of the Runge--Kutta method, is efficiently realized by the automatic differentiation routines provided by the Tensorflow framework. For the numerical quadrature, we choose a composite Gaussian quadrature with $4$ nodes on each subinterval and $10$ subintervals in each direction. 

We now apply the classical Runge--Kutta method of order $4$ and observe the error behaviour for varying step size $h$ and regularization parameter $\varepsilon$.

In Figure~\ref{fig:LV-time-conv}, we fix several values of the regularization parameter $\varepsilon$ and vary the time step size to observe the time convergence behaviour. The $\mathcal H$-norm (i.e. the $L^2$-norm) on $D$ is the natural error measure, which is taken at the fixed time $t=1$. As predicted by the theory, we observe a step size restriction depending on the parameter $\varepsilon$. For smaller values of $\varepsilon$, we require a smaller time step size $h$ in order to achieve convergence. The observed time step restriction is, however, milder than the restriction in Theorem~\ref{prop:glob-err-rk}. On the right-hand side, we visualize the a posteriori bounds for the projections. As expected, these bounds are quite stable with respect to the time step size and estimate the possible accuracy for a fixed parameter $\varepsilon$ and the underlying nonlinear approximation.

In Figure~\ref{fig:LV-eps-conv}, we conversely fix the number of time steps and observe the error and the projection errors for a varying regularization parameter $\varepsilon$. Overall, we observe a convergence of the order of $\mathcal O(\varepsilon)$, when the number of time steps is sufficiently large. When the number of time steps is not sufficiently large, we again observe the effect of the time step restriction. The a posteriori terms of the error bound capture the effects of the regularization parameter $\varepsilon$, but are by construction almost invariant with respect to the time discretization.

\begin{figure}
\includegraphics[width=1.\textwidth]{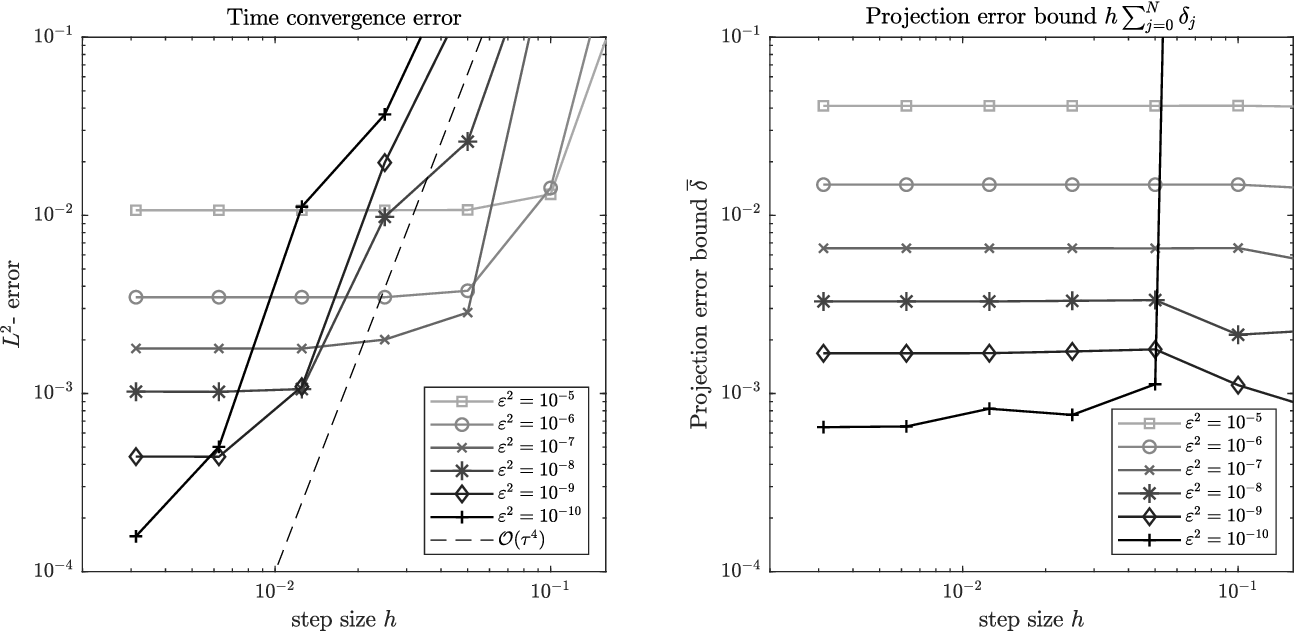}
\caption{Time convergence plot for the Lotka--Volterra system, computed with a fixed neural network architecture with  three hidden layers and four neurons each, which is fully described by $\param\in \mathbb R^{62}$.  We fix the regularization parameter $\varepsilon$ and observe the error behaviour of the classical Runge--Kutta approximation to the regularized flow \eqref{reg-lsq}. On the right-hand side, we plot the projection error term of the error bound described in Theorem~\ref{prop:glob-err-rk}. }
\label{fig:LV-time-conv}
\end{figure}

\begin{figure}
\includegraphics[width=1.\textwidth]{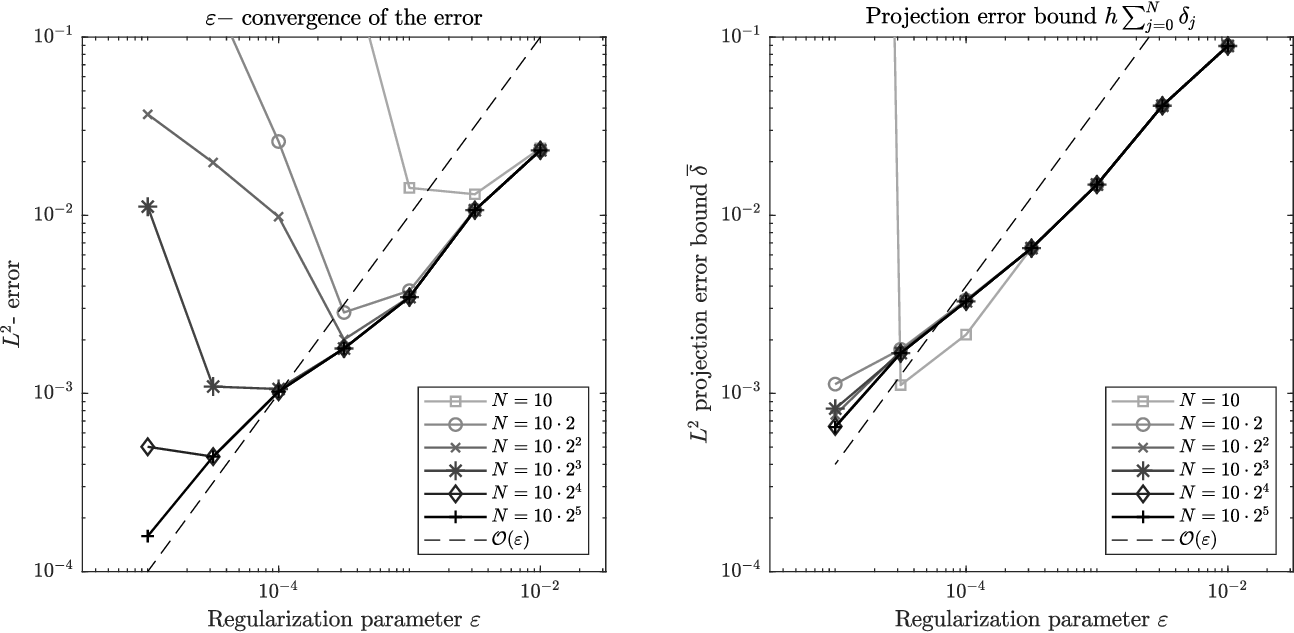}
\caption{The $\varepsilon-$ convergence of the same network architecture, with the same time discretization. We fix the number of time steps and vary the regularization parameter.}
\label{fig:LV-eps-conv}
\end{figure}

\subsection{Approximating double-well quantum dynamics}
We consider a one-dimensional Schr\"odinger equation
$\iu \dot\psi = H\psi$ formulated within the setting of the complex Hilbert space $\calH= L^2(\R,\C)$. 
The equation serves as a model for tunneling dynamics. 
The Schr\"odinger operator
\[
H = -\frac12 \partial_x^2 + \alpha_2 x^2 + \alpha_4 x^4,\quad \psi_0(x) = \pi^{-1/4} \e^{-(x-\param_{\ell})^2/2}, 
\]
contains a quartic double-well potential with polynomial parameters $\alpha_2 = -\frac18$ and $\alpha_4 = \alpha_2^2$.
The initial condition $\psi_0$ is a single normalized Gaussian, whose width stems from the standard harmonic oscillator $H_0 = -\frac12\partial_x^2 + \frac12 x^2$, placed at the left minimum $\param_{\ell} = -2$ of the double well potential. 
During the time interval $[0,T] = [0,12]$ the wave packet travels from the left to the right well, see also \cite{Joubert24}. The approximation ansatz is a frozen sum of $M=36$ Gaussians,
\[
u(t,x) = \sum_{m=1}^M c_m(t) \e^{-x^2/2-\kappa_m(t)x},
\]
with $2M$ complex parameters $\left(c_m(t),\kappa_m(t)\right)\in\paramSpace=\C^2$. For the initialization 
$u_0\in\calM$, we put a non-uniform grid of Gau\ss --Hermite quadrature nodes $(x_i,\xi_j)$ with origin at $(\param_\ell,0)$
on $\R^2$, reformulate the corresponding Gaussian wave packets $\e^{-(x-x_i)^2/2 + \iu \xi_j(x-x_i)}$ in the complex algebraic format $\e^{-x^2/2 - \kappa_m(0)x}$, and determine the optimal linear expansion coefficients 
$c_m(0)$ by solving the linear least squares problem
\[
\left\|u(0) -\psi_0\right\|_\calH  = \min !
\]
The matrices for the initial minimization and the ones involving the parameter Jacobian $\Phi'(\param)$ are evaluated via analytical formulas for Gaussian integrals of the type
\[
\int_\R x^k \e^{-\beta x^2-\lambda x}\,\mathrm{d}x,\qquad \beta>0,\quad \lambda\in\C.
\]
The time integrator is the classical Runge--Kutta scheme of order four. As before, we observe
the error behaviour for varying step size $h$ and regularization parameter $\eps$.

In Figure~\ref{fig:DW-time-conv}, we fix several values of the regularization parameter $\varepsilon$ and vary the time step size $h$ to observe the time convergence behaviour. Since energy
$\mathcal E(\psi(t)) = \langle \psi(t),H\psi(t)\rangle_\calH$, $t\in\R$, is a conserved quantity of the Schr\"odinger evolution, we evaluate 
the approximate energy $\mathcal E(u(t))$ at the final time $t=T$ and compare with the value at initial time $t=0$, see the left-hand side. We observe, that the simulations 
with the smallest regularization parameter $\eps=10^{-5}$ have large errors and even prematurely terminate for 
the particular step size $h=4\cdot 10^{-3}$. For the other choices of the regularization parameter, 
the errors follow the order of the time integrator without the predicted step size restriction. On the right-hand side, we show the a posteriori error bounds for the projections. As before for the Lotka--Volterra model, these bounds are stable with respect to the time step size and estimate the accuracy of the underlying approximation. 

In Figure~\ref{fig:DW-eps-conv}, we conversely fix the size of the time step $h$ and present errors for a varying regularization parameter~$\varepsilon$. The Schr\"odinger dynamics are unitary, and thus conserve the norm $\|\psi(t)\|_\calH$, $t\in\R$, of the solution (mass conservation). On the left hand-side 
we compare the approximate norm $\|u(T)\|_\calH$ at the final time $T$ with the exact unit value. 
We observe decay of the mass error only for relatively large values of the regularization parameter. Depending on the time step size, less regularization results in larger errors. The right-hand side shows 
the a posteriori terms of the error bound. Again, they capture the effects of the regularization parameter $\varepsilon$, but are by construction insensitive with respect to the time discretization.

\begin{figure}
\includegraphics[width=0.49\textwidth]{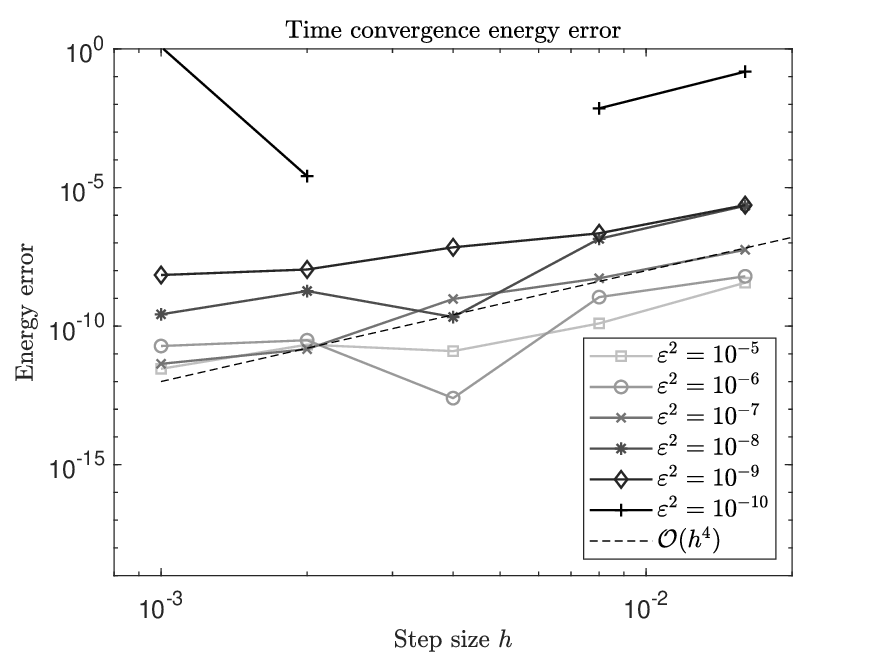}
\includegraphics[width=0.49\textwidth]{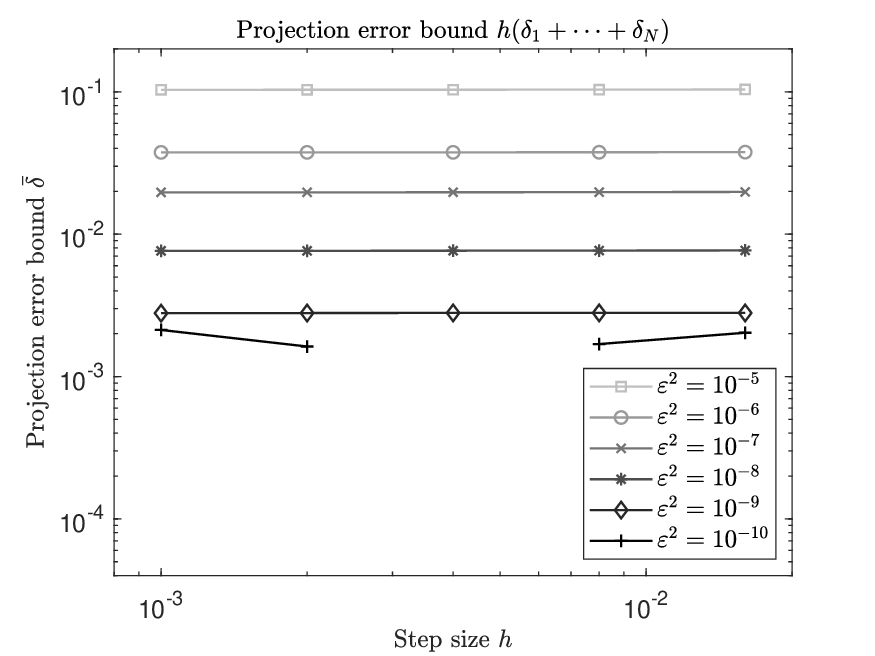}
\caption{Time convergence plot for the double-well system, computed with a sum of $M=36$ complex Gaussians, 
which is fully described by $\param\in \mathbb C^{72}$.  We fix the regularization parameter~$\varepsilon$ and observe the error behaviour of the classical Runge--Kutta approximation to the regularized flow \eqref{reg-lsq}. On the left-hand side, we plot the energy error $\left|\mathcal E(u(T))-\mathcal E(u(0)) \right|$, on the right-hand side the projection error term of the error bound described in Theorem~\ref{prop:glob-err-rk}. }
\label{fig:DW-time-conv}
\end{figure}

\begin{figure}
\includegraphics[width=0.49\textwidth]{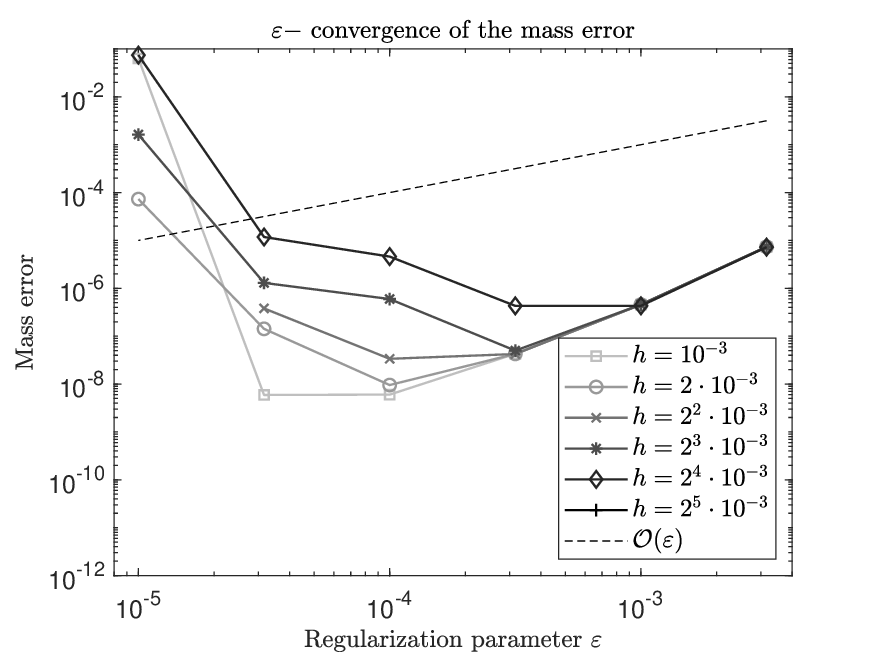}
\includegraphics[width=0.49\textwidth]{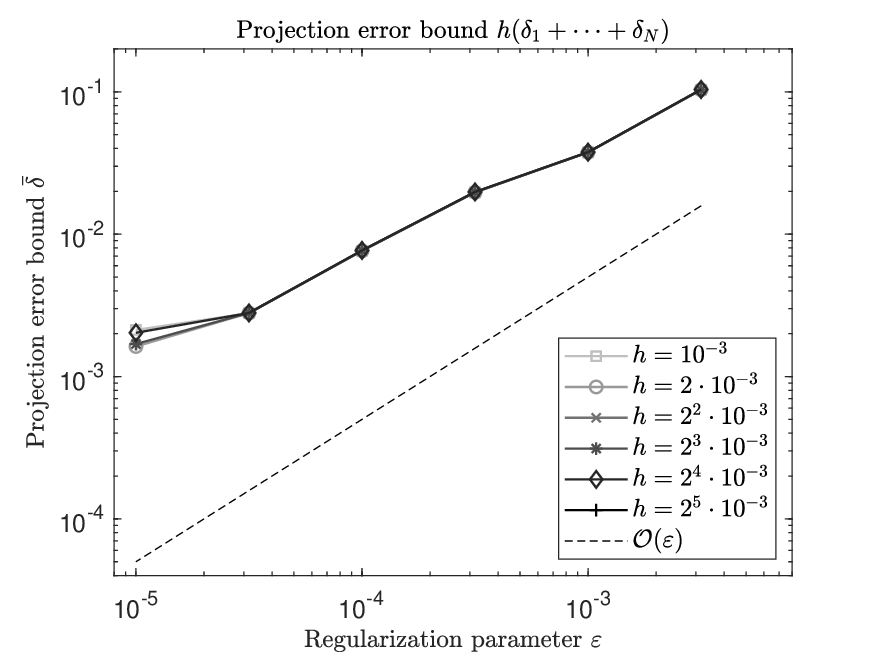}
\caption{The $\varepsilon-$convergence of the same sum of Gaussians, with the same time discretization. We fix the number of time steps and vary the regularization parameter. On the left-hand side, we plot the norm error 
$|\|u(T)\|^2_\calH -1|$, on the right-hand side the projection error.}
\label{fig:DW-eps-conv}
\end{figure}

 	\begin{acknowledgements} 
This work was funded by the Deutsche Forschungsgemeinschaft (DFG, German Research Foundation) under projects SFB 1173 – Project-ID 258734477 and TRR 352 – Project-ID 470903074 as well as the Austrian Science Fund (FWF) under
the special research program Taming complexity in PDE systems (grant SFB F65).
	\end{acknowledgements}

\bibliographystyle{abbrv}
\bibliography{Lit.bib}

%\ \\ \newpage 

\appendix

\section{Appendix: Note on the regularized least squares problem}

To study the sensitivity with respect to the regularization parameter, we consider the linear least squares problem to find $x=x(\alpha) $ such that
$$
\vartheta(\alpha)  := \| Ax-b \|^2 + \alpha \| x \|^2 \quad\text{ is minimal},
$$
where we take the Euclidean norms.
Note that $\vartheta=\delta^2$ and $\alpha= \eps^2$ in the setting of the paper.
The dependence of $\vartheta$ on $\alpha$ is remarkably simple.

\begin{lemma}\label{lem:reg-lsq}
We have 
$$\vartheta'(\alpha) = \| x(\alpha)  \|^2. 
$$
Moreover, 
$$
\vartheta''(\alpha) =\frac{d}{d\alpha} \| x(\alpha)  \|^2\le 0 \quad\text{ and } \quad 
\frac{d}{d\alpha} \, \frac{\vartheta(\alpha) }{\alpha}  \le 0.
$$
\end{lemma}

\begin{proof} With $M=M(\alpha) =(A^\top A + \alpha I)$ we have the normal equations $Mx = A^\top  b$ and hence
$$ 
x=M^{-1}A^\top b.
$$
Since $M' =I$, we have $Mx'+x =0$ and hence
$$
x' = - M^{-1}x.
$$
We have
$$
\vartheta' = 2\langle Ax-b, Ax' \rangle  + 2\alpha \langle x,x' \rangle + \| x \|^2,
$$
which becomes
\begin{align*}
 &\langle Ax-b, Ax' \rangle = \langle Ax-b,  -AM^{-1}x\rangle 
 = -\langle x, A^\top  A M^{-1} x \rangle + \langle M^{-1}A^\top b,x \rangle 
 \\
 &= -\langle x, A^\top  A M^{-1} x \rangle + \langle x,x \rangle
 = \langle x,(I - A^\top  A M^{-1})x \rangle = \langle x, \alpha M^{-1} x \rangle,
\end{align*}
and 
$$
 \alpha \langle x,x' \rangle = \langle x, - \alpha M^{-1} x \rangle, %= \langle x, A^\top  A M^{-1} x \rangle,
$$
so that
$$
\vartheta' = %\langle x,(\alpha I+A^\top A)M^{-1} x \rangle = 
2\langle x,\alpha M^{-1}x\rangle + 2\langle x, -\alpha M^{-1}x \rangle +  \| x \|^2 =
 \|x\|^2,
$$
which is the stated result for the first derivative. The  second derivative is
$$
\vartheta'' = \frac d{d\alpha}\,  \| x(\alpha) \|^2 = 2\langle x,x' \rangle = - 2\langle x, M^{-1}x \rangle,
$$
which is negative (unless $x=0$), since $M$ is positive definite. Finally, 
$$
\left(\frac\vartheta\alpha\right)' = \frac{\alpha \vartheta' - \vartheta}{\alpha^2} 
= \frac{\alpha \| x\|^2  - \vartheta}{\alpha^2}
= - \frac{\|Ax -b \|^2}{\alpha^2},
$$
which is non-positive.
 \qed
\end{proof}

The following matrix estimates are often used in the paper.

\begin{lemma}\label{lem:P-bound} Denote $M_\eps  = A ^\top  A +\eps^2 I$. We have in the matrix 2-norm
\begin{align*}
\|A M_\eps ^{-1}A ^\top \| \le 1,\qquad
\|A M_\eps ^{-1}\| \le \frac{1}{2\eps},\qquad
\|M_\eps ^{-1}\| \le \frac{1}{\eps^2}.
\end{align*}
\end{lemma}

\begin{proof}
The bounds follow from the singular value decomposition $A =U \Sigma V ^\top $ with $U $ a linear isometry, $V $ unitary and $\Sigma $ diagonal. 
Then,  
\begin{align*}
AM^{-1}A^\top  = U\Sigma\left(\Sigma^2+\eps^2I\right)^{-1}\Sigma U^\top .
\end{align*}
Since $U $ is a linear isometry, we have
%\[
%\|U p\|_\calH = \|p\|_\paramSpace\quad\text{and}\quad
%\|\Upsilon(\param)^* w\|_\paramSpace \le \|w\|_\calH
%\]
%for all $p\in\paramSpace$ and all $w\in\calH$. Therefore, for any $A\in L(\paramSpace)$,
%\begin{align*}
%    \|\Upsilon(\param) A\Upsilon(\param)^*\|_{L(\calH)} &= 
%    \sup_{w\neq 0} \frac{\|\Upsilon(\param)A\Upsilon(\param)^*w\|_\calH}{\|w\|_\calH}\\
%    &\le \sup_{w\neq 0} \frac{\|A\Upsilon(\param)^*w\|_\paramSpace}{\|\Upsilon(\param)^*w\|_\paramSpace}
%    \le \|A\|_{L(\paramSpace)}.
%\end{align*}
%sing this estimate for $A=\Sigma(\param)\left(\Sigma(\param)^2+\eps^2I_\paramSpace\right)^{-1}\Sigma(\param)$, we obtain
\begin{align*}
\|A M_\eps ^{-1}A ^\top \|
&\le \|\Sigma \left(\Sigma ^2+\eps^2I\right)^{-1}\Sigma \|
\le\  \sup_{\sigma\ge 0} \frac{\sigma^2}{\sigma^2+\eps^2} = 1.
\end{align*}
Similarly, $AM^{-1} = U\Sigma(\Sigma^2+\eps^2I)^{-1}U^\top$, so that
\begin{align*}
\|A M_\eps ^{-1}\| 
&= \|\Sigma (\Sigma ^2 +\eps^2)^{-1}\|
\le \sup_{\sigma\ge 0} \frac{\sigma}{\sigma^2+\eps^2} = \frac{1}{2\eps},
\end{align*}
and further
$$
\|M_\eps ^{-1}\| 
= \| (\Sigma ^2 +\eps^2)^{-1}\|
\le \sup_{\sigma\ge 0} \frac{1}{\sigma^2+\eps^2} = \frac{1}{\eps^2},
$$
as stated. \qed
\end{proof}

\end{document}